\documentclass[11pt]{amsart}

\usepackage{amstext}
\usepackage{amsthm}
\usepackage{amssymb}
\usepackage[a4paper, total={6in, 9in}]{geometry}
\usepackage{enumerate}
\usepackage{xcolor}
\usepackage{comment}

\newtheorem{theorem}{Theorem}[section]
\newtheorem{lemma}[theorem]{Lemma}
\newtheorem{proposition}[theorem]{Proposition}
\newtheorem{corollary}[theorem]{Corollary}
\newtheorem{remark}[theorem]{Remark}

\DeclareMathOperator{\loc}{loc}

\DeclareMathOperator{\dist}{dist}

\DeclareMathOperator{\tr}{tr}
\DeclareMathOperator{\sym}{Sym}
\DeclareMathOperator{\Cyl}{Cyl}

\DeclareMathOperator{\graph}{graph}

\DeclareMathOperator{\cl}{cl}

\title[Uniqueness of convex ancient solutions]{Uniqueness of convex ancient solutions to hypersurface flows}
\author{Stephen Lynch}

\begin{document}

\begin{abstract}
We show that every convex ancient solution of mean curvature flow with Type I curvature growth is either spherical, cylindrical, or planar. We then prove the corresponding statement for flows by a natural class of curvature functions which are convex or concave in the second fundamental form. Neither of these results assumes interior noncollapsing. 
\end{abstract}

\maketitle

\section{Introduction}

A basic problem in geometric analysis is to determine the possible shapes of ancient solutions to geometric flows. Such solutions, called ancient because they extend arbitrarily far back in time, arise naturally as limits of parabolic dilations at  singularities. For this reason ancient solutions are an indispensable tool in understanding singularity formation, and have consequently played a central role in applications such as Perelman's proof of the geometrisation conjecture.

This work establishes uniqueness results for ancient families of smooth hypersurfaces $M_t$ in $\mathbb{R}^{n+1}$ that are \textbf{convex}, meaning $M_t = \partial \Omega_t$ for $\Omega_t$ open and convex, and solve mean curvature flow, or else move by a more general curvature function satisfying natural conditions. All of the evolution equations we consider are invariant under parabolic rescalings. In this setting, following Hamilton \cite{Ham93}, ancient solutions with bounded second fundamental form $A$ are typically divided into two families: $\{M_t\}_{t\in(-\infty,T]}$ is said to be of \textbf{Type I} if
\begin{equation}
\label{eq:intro_TypeI}
\limsup_{t\to -\infty} \sup_{M_t} \sqrt{-t} \, |A| < \infty,
\end{equation}
whereas solutions for which \eqref{eq:intro_TypeI} fails are of \textbf{Type II}. Our first result classifies ancient solutions of mean curvature flow that are convex and Type I.
\begin{theorem}
\label{thm:main_MCF}
Every Type I convex ancient solution of mean curvature flow in $\mathbb{R}^{n+1}$, $n \geq 2$, is either a family of homothetically shrinking spheres or cylinders, or else consists of stationary hyperplanes. 
\end{theorem}

Here and throughout, a cylinder is a round $S^{n-m} \times \mathbb{R}^m$, $1 \leq m \leq n-1$. It is interesting to compare Theorem \ref{thm:main_MCF} with Hamilton's differential Harnack inequality \cite{Ham_har}, which implies every Type II convex ancient solution of mean curvature flow is asymptotic to a translating soliton (cf. \cite{Huisk-Sin15}[Theorem 4.7]).

Consider now a maximal solution of mean curvature flow $\{M_t\}_{t\in[0,T)}$ such that $M_0$ is compact, embedded and has nonnegative mean curvature. By work of Huisken-Sinestrari \cite{Huisk-Sin99a} and White \cite{White00}, \cite{White03} (see also \cite{Hasl-Klein}), any ancient solution arising as a blow-up limit of the flow at $T$ is convex. Thus, if the limit is also Type I, it is spherical, cylindrical or planar by Theorem \ref{thm:main_MCF}. This includes all limits arising at a Type I singularity, that is when the bound $ |A| \leq C/\sqrt{T-t}$ holds for all $t <T$ and some $C <\infty$.

In a similar way Theorem \ref{thm:main_MCF} can also be used to study singularities in solutions that are only immersed and mean-convex. For this purpose it is crucial the theorem does not require interior noncollapsing (in the sense of \cite{Sheng-Wang} and \cite{And12}), since blow-up limits of immersed mean-convex solutions are not known to have this property (except under uniform two-convexity \cite{Naff_models}). The Huisken-Sinestrari convexity estimate \cite{Huisk-Sin99a} does apply to immersed mean-convex solutions, implying smooth blow-up limits have nonnegative second fundamental form, and thus are either convex or have zero scalar curvature by a result of Sacksteder \cite{Sacksteder} and the strong maximum principle. Theorem \ref{thm:main_MCF} also sheds new light on the formation of singularities in higher codimensions, where embeddedness is not preserved and there is no notion of interior noncollapsing, but sufficient pinching implies blow-ups are codimension-one \cite{Naff} and convex \cite{Lynch_Nguyen}.

Our result adds to a host of uniqueness theorems for ancient solutions of mean curvature flow under various hypotheses. In particular, Huisken-Sinestrari established the compact case of Theorem \ref{thm:main_MCF} in \cite{Huisk-Sin15}, along with other uniqueness results for the shrinking sphere. Haslhofer-Hershkovits proved similar results under an additional noncollapsing assumption \cite{Haslhofer_Hershkovits}. Self-shrinking solutions that are embedded and mean-convex were classified by Huisken \cite{Huisk93} and Colding-Minicozzi \cite{CM12} - these are all spheres and cylinders. Brendle proved that every compact genus-zero self-shrinking embedding in $\mathbb{R}^3$ is a round sphere \cite{Brendle_shrink}. Daskalopoulos-Hamilton-Sesum \cite{Dask-Ham-Ses} showed that every compact convex ancient solution in $\mathbb{R}^2$ is either a shrinking circle or an Angenent oval, and Bourni-Langford-Tinaglia \cite{BTL20} recently extended this classification to the noncompact case, showing the only additional examples that arise there are the grim reaper and stationary lines. 

A remarkable series of recent works have led to a complete classification of convex ancient solutions that are uniformly two-convex (this is implied by convexity when $n=2$) and interior noncollapsed. By work of Brendle-Choi \cite{Brendle-Choi19}, \cite{Brendle-Choi18} the only noncompact solutions in this class are cylinders with one Euclidean factor and the translating bowl soliton, and Angenent-Daskalopoulos-Sesum \cite{ADS20} showed the only compact possibilities are spheres and the $O(n)\times O(1)$-symmetric ovaloid. These works build on an earlier uniqueness result for the bowl amongst translators due to Haslhofer \cite{Hasl_bowl}.

The proof of Theorem \ref{thm:main_MCF} is by dimension-reduction, and broadly proceeds as follows. Given a Type I convex ancient solution $M_t = \partial \Omega_t$, $t \in(-\infty,T]$, the strong maximum principle implies $M_t$ is either planar or has positive mean curvature. In the latter case we establish $\Omega_t$ contains a ball of radius $\sim\sqrt{-t}$ centered at the origin. This enables us to extract a blow-down limit of $M_t$ that sweeps out all of space as $t$ ranges over $(-\infty,0)$, and thus splits off a Euclidean factor by a result of Xu-Jia Wang. By carefully exploiting the Type I property we show this Euclidean factor coincides with the tangent cone at infinity of the original solution. From this, we conclude $M_t$ splits as a Euclidean factor times a compact convex solution of Type I which is then spherical by \cite{Huisk-Sin15} and \cite{Dask-Ham-Ses}.

\subsection{Fully nonlinear flows}

We turn now to parabolic evolution equations that deform hypersurfaces by more general curvature functions. Mean curvature flow is the unique quasilinear equation in this class which is invariant under ambient isometries - all others are fully nonlinear. Important examples include the flow of convex hypersurfaces by their Gauss curvature introduced by Firey \cite{Firey} as a model for the wearing of tumbling stones. Flows by powers of the Gauss curvature are also of importance, arising for example in affine geometry \cite{Andrews_affine} and in connection with convex hull peeling \cite{Calder_Smart}. The asymptotic behaviour of these flows has been described completely \cite{Chow}, \cite{Andrews_affine}, \cite{Andrews_Firey}, \cite{Brendle_Choi_Dask}.  

Andrews initiated the study of flows by general one-homogeneous curvature quantities that are convex or concave in \cite{And94_euclid}. Concave flows of this type have been employed to prove deep results on the topology of hypersurfaces in Riemannian manifolds, first for positive principal curvatures in \cite{And94}, and later under strict two-convexity \cite{Bren-Huisk17}. Mean curvature flow is unsuitable for these applications since it fails to preserve the curvature conditions considered (except in highly symmetric ambient spaces), in contrast to the fully nonlinear equations used by the authors. We note \cite{Bren-Huisk17} contains the first construction of a fully nonlinear geometric flow with surgeries. 

Despite significant progress (see for example \cite{And-Lang-McCoy13}, \cite{And-Lang-McCoy14}, \cite{Bren-Huisk17}, \cite{Lang-Lyn20}, \cite{Lynch20}), much remains unclear about singularities in non-convex solutions of fully nonlinear flows. This is due to the fact that certain ubiquitous tools in mean curvature flow, such as Huisken's monotonicity formula \cite{Huisken90}, generally speaking have no analogue in the fully nonlinear setting. Our second main result, a generalisation of Theorem \ref{thm:main_MCF} to flows by a large class of one-homogeneous concave and convex curvature functions, fills part of this technical gap.

Our setting is as follows. We consider families of hypersurfaces $\{M_t\}_{t\in I}$ such that $M_t = F(M,t)$ for an evolving immersion $F:M\times I \to \mathbb{R}^{n+1}$, $n\geq2$, which solves $\partial_t F = - G\nu$. Here $\nu$ is the outward unit normal and $G = \gamma(\lambda)$, where $\lambda$ denotes the principal curvatures of $M_t$, and $\gamma\in C^\infty(\Gamma)$ is an \textbf{admissible speed}; that is, $\gamma$ is positive, increasing in each argument and homogeneous of degree one, and $\Gamma$ is an open, symmetric, convex cone in $\mathbb{R}^n$. An admissible speed $\gamma$ is {\bf inverse-concave} if $\Gamma$ contains the positive cone $\Gamma_+ := \{\lambda_i >0\}$ and the function
\[\gamma_*(\lambda) := \gamma( \lambda_1^{-1},\dots,\lambda_n^{-1} )^{-1}, \qquad \lambda \in \Gamma_+,\]
is concave. A family $\{M_t\}_{t\in I}$ moving with speed $G=\gamma(\lambda)$ will be referred to as a $\gamma$-flow. A $\gamma$-flow is {\bf uniformly parabolic} if its normalised principal curvatures $\lambda / G$ lie in a fixed compact subset of $\Gamma$ for all times. 

We say $\gamma$ admits a \textbf{splitting theorem} if every uniformly parabolic convex ancient $\gamma$-flow is either strictly convex ($A>0$) or splits as a product $M_t := \mathbb{R}^m \times M_t^\perp$, where $M_t^\perp$ is strictly convex in $\mathbb{R}^{n-m+1}$. All convex admissible speeds admit a splitting theorem \cite{And-Lang-McCoy14}, and we prove a splitting theorem for a natural class of concave speeds in Appendix \ref{app:splitting}. The latter applies to the two-harmonic mean studied in \cite{Bren-Huisk17} and the generalised $k$-harmonic means studied in \cite{Lang-Lyn20}.

\begin{theorem}
\label{thm:main_FNL}
Consider an admisssible speed $\gamma$ which is either convex, or concave and inverse-concave, and admits a splitting theorem. Every uniformly parabolic convex ancient $\gamma$-flow of Type I is a family of homothetically shrinking spheres or cylinders.
\end{theorem}

Let us remark on the assumptions of the theorem. Assuming $\gamma$ is convex or concave is standard, since in this case the estimates of Krylov-Safonov \cite{Kryl-Saf} and Krylov \cite{Kryl82} (see also \cite{Ev82}) ensure the PDE governing the flow have good regularity properties. Flows by concave speeds do not always preserve convexity \cite{AMY}, but flows by inverse-concave speeds do \cite{And07}, and this property is required for our arguments. Uniform parabolicity and convexity are known to hold for blow-up limits of the flow if $\gamma$ is convex \cite{And-Lang-McCoy14}, or if $\gamma$ is one of the concave speeds considered in \cite{Bren-Huisk17} or \cite{Lynch20}.\footnote{We expect further concave examples will be added to this list.} Theorem \ref{thm:main_FNL} yields a classification of Type I singularities for these flows. 

We established the compact case of Theorem \ref{thm:main_FNL} with Langford in \cite{Lang-Lyn20} (cf. \cite{Risa-Sinestrari}). There is also a classification result for uniformly-two convex translators moving by nonlinear speeds due to Bourni-Langford \cite{Bourni-Lang}. Theorem \ref{thm:main_FNL} ought to be a stepping stone to obtaining further uniqueness theorems for ancient solutions - for example, the corresponding statement for Ricci flow was employed in Brendle's groundbreaking work \cite{Brendle20}. Furthermore, while the proof of Theorem \ref{thm:main_FNL} proceeds by a dimension-reduction scheme very similar to that used in Theorem \ref{thm:main_MCF}, implementing this strategy in the fully nonlinear setting requires the development of some new tools. These may be useful in other contexts. 
 
\subsection{Outline} In Section \ref{sec:local_curvature_ests} we use interior curvature estimates to prove a Liouville-type theorem for convex ancient $\gamma$-flows. The result, which plays an important role in the proofs of Theorem \ref{thm:main_MCF} and Theorem \ref{thm:main_FNL}, asserts that the tangent cone at infinity of any uniformly parabolic convex ancient solution lies in a hyperplane. The proof of Theorem \ref{thm:main_MCF} is then carried out in Section \ref{sec:MCF}, and that of Theorem \ref{thm:main_FNL} in Section \ref{sec:FNL}. Appendix \ref{app:splitting} contains proof of a splitting theorem valid for a natural class of flows by concave speeds. 

\section{Interior curvature estimates}
\label{sec:local_curvature_ests}

We begin with an interior curvature estimate for $\gamma$-flows that can be expressed locally as a radial graph. This is \cite{Bren-Huisk17}[Proposition 5.1], but modified slightly to be sharper at large distances. The estimate in \cite{Bren-Huisk17} in turn takes inspiration from an estimate of Tso for convex solutions of the Gauss curvature flow \cite{Tso}, and the Ecker-Huisken interior curvature bound for graphical solutions of mean curvature flow \cite{Eck-Huisk}.

In this section $\gamma$ always denotes an admissible speed, but we do not need to assume convexity or concavity. The following result applies in particular to solutions of mean curvature flow, but in this case we take $\Gamma = \mathbb{R}^n$, so uniform parabolicity is vacuous.

\begin{proposition}
\label{prop:graph_bd}
Fix an admissible speed $\gamma \in C^\infty(\Gamma)$. Let $\{M_t\}_{t\in [-K^2r^2,0]}$ be a uniformly parabolic $\gamma$-flow such that each $M_t$ is the boundary of a connected open set $\Omega_t$. Assume $\Omega_t$ contains $B(p,r)$ for all $ t\in [-K^2 r^2,0]$ and 
\[\inf_{B(p,2Lr) \times [-K^2r^2,0]} \nu(x,t) \cdot \frac{x - p}{|x - p|} \geq \theta,\]
where $L>1$ and $\theta >0$. Then we have
\begin{equation}
\label{eq:graph_bd}
\sup_{B(p, Lr) \times [-K^2 r^2, 0]} \;(L^2r^2 - |x-p|^2 ) (t+K^2r^2)^\frac{1}{2} G \leq C (1+K)L^2 \theta^{-3} r^2
\end{equation}
where $C = C(n, \gamma, \dist(\lambda/G, \partial \Gamma))$.
\end{proposition}
\begin{proof}
Since the proof is very similar to that of Proposition 5.1 in \cite{Bren-Huisk17}, we only provide a sketch. Without loss of generality we may assume $p =0$. For each point $x \in M_t\cap B(0,Lr)$, define
\[\eta = L^2r^2 - |x|^2, \qquad w = \frac{(\nu\cdot x)^2}{|x|^2} - \frac{\theta^2}{2},  \qquad v=w^{-\frac{1}{2}}.\]
Note $w \geq \tfrac{\theta^2}{2}$ in $B(0,2Lr)$, so $v$ is well defined. The theorem is established by applying the maximum principle to the quantity $\psi := \eta v G$. Throughout the proof $C$ is a large constant depending only on $n$, $\gamma$ and a global lower bound for $\dist(\lambda/G, \partial\Gamma)$.

By classical theorems for smooth symmetric functions \cite{Glaeser}, \cite{Schwarz}, in any local orthonormal frame $G$ can be viewed interchangeably as a smooth function of the principal curvatures or the components of the second fundamental form. Writing $\dot \gamma^{ij}:= \tfrac{\partial \gamma}{\partial A_{ij}}$, we have the evolution equations
\begin{align*}
(\partial_t - \dot \gamma^{ij} \nabla_i\nabla_j ) |x|^2 ={} - 2 \dot \gamma^{ij} g_{ij},\\
(\partial_t - \dot \gamma^{ij}\nabla_i\nabla_j )  (\nu\cdot x) ={} \dot \gamma^{ij} A^2_{ij} (\nu\cdot x) - 2G,
\end{align*}
where $g$ denotes the induced metric. From these we compute
\begin{align*}
(\partial_t - \dot \gamma^{ij}\nabla_i\nabla_j) w = {}& 2(\dot \gamma^{ij}A^2_{ij} - 2(x\cdot \nu)^{-1}G)  (w+\theta^2/2)- \frac{1}{2(w + \theta^2/2)} \dot \gamma^{ij}\nabla_iw\nabla_jw\\
& +\dot \gamma^{ij}\nabla_i w \frac{\nabla_j |x|^2}{|x|^2} +\frac{(w + \theta^2/2)}{|x|^2}\bigg( 2\dot \gamma^{ij} g_{ij}- \dot \gamma^{ij}\frac{\nabla_i|x|^2\nabla_j|x|^2}{2|x|^4}\bigg).
\end{align*}
Since $|\nabla |x|^2| \leq 2|x|$ we observe the last term on the right is nonnegative, and with Young's inequality we obtain
\begin{align*}
(\partial_t - \dot \gamma^{ij}\nabla_i\nabla_j) w \geq{}&2(\dot\gamma^{ij} A_{ij}^2 - 2(x\cdot \nu)^{-1}G)  (w+\theta^2/2)- \frac{1 + \theta^2/4}{2(w + \theta^2/2)} \dot\gamma^{ij}\nabla_iw\nabla_jw\\
& -C\theta^{-2}r^{-2} (w + \theta^2/2).
\end{align*}
We have $\dot \gamma^{ij} \geq C^{-1} g_{ij}$ and $ x \cdot \nu \geq \theta r$, so using Young's inequality again we find
\[\dot\gamma^{ij}A_{ij}^2 - 2(x\cdot\nu)^{-1}G \geq C^{-1}G^2 - C\theta^{-2}r^{-2},\]
and thus obtain
\begin{align*}
(\partial_t - \dot \gamma^{ij}\nabla_i\nabla_j) w \geq{}&2(C^{-1}G^2-C\theta^{-2}r^{-2})  (w+\theta^2/2)- \frac{1 + \theta^2/4}{2(w + \theta^2/2)} \dot\gamma^{ij}\nabla_iw\nabla_jw\\
& -C\theta^{-2}r^{-2} (w + \theta^2/2).
\end{align*}

Using the evolution equations
\[(\partial_t-\dot\gamma^{ij}\nabla_i\nabla_j)\eta = 2\dot \gamma^{ij} g_{ij}, \qquad (\partial_t -\dot\gamma^{ij}\nabla_i\nabla_j)G = \dot\gamma^{ij}A^2_{ij} G\]
we now proceed exactly as in \cite{Bren-Huisk17}[Proposition 5.1] to find
  \begin{align}
  \label{eq:interior_evol}
(\partial_t - \dot\gamma^{ij}\nabla_i \nabla_j)\psi \leq &  - C^{-1} \theta^2 \eta^{-2} \psi^3 + C L^2 \theta^{-1}  r \eta^{-2} \psi^2+CL^4 \theta^{-2}r^{2} \eta^{-2}\psi \notag\\
&  -2 \dot\gamma^{ij} (\eta^{-1} \nabla_i\eta +v^{-1} \nabla_i v) \nabla_j\psi.
\end{align} 
Therefore, if $(\bar x, \bar t)$ is chosen to satisfy
\[(\bar t +K^2 r^2)^\frac{1}{2} \psi( \bar x, \bar t) := \sup\Big\{(t+K^2r^2)^\frac{1}{2} \psi(x,t) : x \in B(0,Lr), \; t \in [-K^2 r^2, 0]\Big\},\] 
evaluating \eqref{eq:interior_evol} at $(\bar x, \bar t)$ gives
\begin{align*}
C^{-1}\theta^2  \psi^3\leq {}&  CL^2\theta^{-1} r \psi^2  +CL^4 \theta^{-2} r^2  \psi+ \frac{1}{2}(\bar t+K^2r^2)^{-1}\eta^2\psi.
\end{align*}
If $\psi(\bar x, \bar t) \geq 10 C^2 L^2 \theta^{-3}r $ we obtain
\[(\bar t +K^2r^2) \psi(\bar x, \bar t)^2 \leq C \theta^{-2} \eta^2 \leq  C L^4\theta^{-2} r^4.\]
If on the other hand $\psi(\bar x, \bar t) \leq 10 C^2 L^2 \theta^{-3}r$, then
\begin{align*}
(\bar t + K^2 r^2)^\frac{1}{2} \psi(\bar x, \bar t) \leq CKL^2 \theta^{-3} r^2.
\end{align*}
The claim follows by combining these two inequalities.
\end{proof}

The boundary of any smooth convex domain is locally a radial graph, so we obtain the following corollary of Proposition \ref{prop:graph_bd} for convex solutions.

\begin{corollary}
\label{cor:conv_curv_bd}
Fix an admissible speed $\gamma\in C^\infty(\Gamma)$. Suppose $\{M_t =\partial \Omega_t\}_{t\in [-K^2r^2,0]}$ is a uniformly parabolic convex $\gamma$-flow. Then if $B(p,r) \subset \Omega_0$, given $L >1$ we have
\begin{equation*}
\sup_{B(p, Lr) \times [-K^2 r^2, 0]}  \;(L^2 r^2 - |x-p|^2 ) (t+K^2r^2)^\frac{1}{2}  G \leq C (1+K) L^5 r^2,
\end{equation*}
where $C=C(n,\gamma,\dist(\lambda/G, \partial \Gamma))$.
\end{corollary}

\begin{proof}
Since the sets $\Omega_t$ move inward with increasing $t$, $B(p,r)$ is contained in $\Omega_t$ for all $t \leq 0$. If $x \in M_t$, since $\Omega_t$ is convex, $\cl(\Omega_t)$ contains the convex hull of $x$ and the equator in $\partial B(p,r)$ orthogonal to $x-p$. From this observation we deduce, for $L>1$ and $|x-p|\leq 2Lr$,
\[\nu(x,t) \cdot \frac{x-p}{|x-p|} \geq \frac{r}{\sqrt{ r^2 + |x-p|^2}} \geq \frac{1}{\sqrt{1+4L^2}} ,\]
so inside $B(p,2L r)$ we may apply Proposition \ref{prop:graph_bd} with $\theta = \tfrac{1}{3}L^{-1}$. 
\end{proof}

\subsection{Tangent cone at infinity}

Let $\Omega$ be a convex domain in $\mathbb{R}^{n+1}$. For any positive sequence $a_j \to 0$, the sequence of rescaled domains $a_j \Omega$ converges to a closed convex cone in the Hausdorff topology. We refer to the limiting cone, which is independent of the sequence $a_j$, as the {\bf tangent cone at infinity} of $\Omega$, and denote it $\mathcal T_\infty \Omega$. If $\Omega$ is noncompact, for each $p \in \Omega$ the set $p+\mathcal T_\infty \Omega$ is the largest closed cone with vertex at $p$ which lies in the closure of $\Omega$. If $\mathcal T_\infty \Omega$ is compact then $\mathcal T_\infty \Omega = \{0\}$ and $\Omega$ is precompact.

For a convex set $\Omega \in \mathbb{R}^{n+1}$, $\dim \Omega $ is defined to be the dimension of the smallest affine subspace of $\mathbb{R}^{n+1}$ containing $\Omega$. Using Corollary \ref{cor:conv_curv_bd} we deduce that the tangent cone at infinity of a uniformly parabolic convex ancient $\gamma$-flow cannot be $(n+1)$-dimensional. 

\begin{proposition}
\label{prop:asymptotic_cone}
Fix an admissible speed $\gamma \in C^\infty(\Gamma)$. If $\{M_t=\partial \Omega_t\}_{t\in(-\infty,0]}$ is a uniformly parabolic convex $\gamma$-flow then $\dim\mathcal T_\infty \Omega_0 \leq n$.  
\end{proposition}

\begin{proof}
Translate so $0 \in M_0$ and perform a parabolic rescaling to arrange $G(0,0)=1$. Suppose for a contradiction $ \dim \mathcal T_\infty \Omega_0 =n+1$. Then there is an open ball $B(p,\delta) \subset \mathcal T_\infty \Omega_0$. To ease notation, suppose $\delta =1$ (otherwise, we can scale the ball). Since $\mathcal T_\infty \Omega_0$ is a cone, for each $r \geq 0$ we have 
\[B(r p, r) \subset \mathcal T_\infty\Omega_0\subset \cl(\Omega_0).\]
Setting $L := 4\hspace{0.5mm}|p|$ ensures the balls $B(r p, \tfrac{Lr}{2})$ sweep out $\mathbb{R}^{n+1}$ as $r \to \infty$, but Corollary \ref{cor:conv_curv_bd} implies
\[G \leq C L^3 r^{-1}\]
inside $B(r p,\tfrac{Lr}{2}) \times [-\tfrac{L^2r^2}{4},0]$, where $C$ depends on $n$, $\gamma$ and the quality of the uniform parabolicity. In particular, for all sufficiently large $r$, 
\[1 = G(0,0) \leq C L^3 r^{-1},\]
which is the desired contradiction. 
\end{proof}

We note Proposition \ref{prop:asymptotic_cone} fails dramatically if the solution is not assumed to be ancient. For example, there are plenty of expanding solutions of mean curvature flow which are convex and asymptotic to an $(n+1)$-dimensional cone \cite{Eck-Huisk}.

\section{Mean curvature flow}
\label{sec:MCF}

In this section we establish Theorem \ref{thm:main_MCF}. To that end, let $\{M_t=\partial \Omega_t\}_{t\in(-\infty,0]}$ be a Type I convex ancient solution of mean curvature flow in $\mathbb{R}^{n+1}$, $n \geq 2$. To be clear we assume there is a global parameterisation $F:M\times(-\infty,0] \to \mathbb{R}^{n+1}$ such that $M_t = F(M,t)$ and $\partial_t F = - H\nu$, where $\nu$ is the outward normal to $M_t$. Every convex ancient solution with bounded second fundamental form admits such a parameterisation, so this does not constitute a loss of generality. 

If the mean curvature $H$ of $M_t$ vanishes somewhere the strong maximum principle implies $\{M_t\}_{t\in(-\infty,0]}$ is either a stationary hyperplane or pair of hyperplanes, so let us assume $H >0$. Since $\{M_t\}_{t\in(-\infty,0]}$ is Type I there is a time $T\leq 0$ and a constant $K<\infty$ such that
\[\sup_{t \leq T} \,\sup_{M_t} \sqrt{-t} \, H \leq \frac{K}{2}.\] 
To establish Theorem \ref{thm:main_MCF} we need to show $\{M_t\}_{t\in(-\infty,0]}$ is a family of homothetically shrinking spheres or cylinders. This is true if and only if $\{M_{t+T}\}_{t\in(-\infty,-T]}$ is a family of homothetically shrinking spheres or cylinders, so without loss of generality assume
\begin{equation}
\label{eq:Type I MCF}
\sup_{t\leq 0} \, \sup_{M_t} \sqrt{-t} \, H \leq \frac{K}{2}.
\end{equation}
Shifting the solution in space, we can also arrange $0 \in M_0$.

We begin by integrating the Type I property to obtain an upper bound for the displacement of $M_t$ from $\Omega_0$. Given a set $X \in \mathbb{R}^{n+1}$ and a constant $R\geq 0$, let us write
\[D(X,R) := \{p\in\mathbb{R}^{n+1}:\dist(p,X)\leq R\}.\]

\begin{lemma}
\label{lem:type-I_envelope}
For every $t\leq 0$, 
\[M_t \subset D(\Omega_0 , K\sqrt{-t}).\]
\end{lemma}
\begin{proof}
Each $x \in M_t$ can be written as $F(p,t)$ for some $p \in M$. Using \eqref{eq:Type I MCF} we find that
\begin{align*}
|F(p,0) - F(p,t) | &\leq \int_{t}^{0} |H(p,\tau)|\, d\tau \leq K\sqrt{-t},
\end{align*}
hence 
\[\dist(x, \Omega_0) = \dist(F(p,t), \Omega_{0}) \leq K\sqrt{-t}.\]
\end{proof}

Next, we show the mean curvature of $M_t$ does not fall off too quickly near the origin as $t \to -\infty$. 

\begin{lemma}
\label{lem:lower-speed_MCF}
For each $L>0$ there is a positive constant $c_L$ such that
\[\sqrt{1 -t}\, H(x,t) \geq c_L \]
for every $t\leq 0$ and $x\in M_t \cap B(0,L\sqrt{-t} )$. 
\end{lemma}

\begin{proof}
We argue by contradiction. The hypersurfaces $M_t$ are strictly mean-convex, so if the claim is false for some $L>0$ there is a sequence of times $t_j \to -\infty$ and points $x_j \in M_{t_j} \cap B(0, L\sqrt{-t_j})$ such that 
\[\sqrt{-t_j}\,H(x_j,t_j) \to 0.\]
Let us write $a_j = 1/\sqrt{-t_j}$ and define a sequence of rescaled solutions
\[ \{M_t^j:=a_j M_{a_j^{-2} (t-1)}\}_{t \in (-\infty,0]}.\] 
By Lemma \ref{lem:type-I_envelope}, 
\[M_{0}^j = a_j M_{-a_j^{-2}} \subset a_j D(\Omega_0, Ka_j^{-1}) = D(a_j \Omega_0, K).\]
We use an upper index $j$ to indicate quantities related to $M_t^j$, so $H^j$ is the mean curvature of $M_t^j$, etc. Convexity implies $|A^j|\leq H^j$, so by the Type I property $|A^j| \leq K$ at times $t \leq 0$. Set $y_j := a_j x_j$.

Let us pass to a subsequence so $\Omega_{0}^j$ converges to a closed convex limit $\mathcal K$ with respect to the Hausdorff topology. We use a compactness argument and the strong maximum principle to show $H^j(y_j,0) \to 0$ implies $\mathcal K$ contains a hyperplane, as follows.

Since $y_j \in B(0,L)$ we may pass to a subsequence to ensure $y_j$ and $\nu^j(y_j,0)$ both converge as $j \to \infty$. Applying a rigid motion in the ambient space, we may assume 
\[y_j \to 0, \qquad \nu^j(y_j,0) \to e_{n+1}.\]
Let $U^j$ denote the largest connected neighbourhood of $y_j$ in $M_0^j$ in which $\nu^j \cdot e_{n+1} > \tfrac{1}{2}$, and write $P$ for the orthogonal projection from $\mathbb{R}^{n+1}$ to $\{x_{n+1} = 0\}$. The bound $|A^j|\leq K$ implies there is a $\rho' = \rho'(K)$ such that 
\[B_{\rho'} \subset P(U^j)\]
for all large $j$, where $B_r:=B(0,r) \cap \{x_{n+1} =0\}$, so by the inverse function theorem $P^{-1}(B_{\rho'}) \cap U^j$ is the graph of a smooth function $u^j : B_{\rho'} \to \mathbb{R}$. Define
\[Z_r^j := P^{-1}(B_r) \cap\{ x_{n+1} \geq u^j(x_1, \dots, x_n)\}.\]
Then since $\Omega_t^j$ is convex, $M_t^j \cap Z_r^j$ is a graph over $B_{r}$ for all $r < \rho'$ and $t\leq 0$. Moreover, there is a $\rho = \rho(K)\leq \rho'$ such that $\nu^j \cdot e_{n+1} \geq \tfrac{1}{2}$ in $M_t^j \cap Z_\rho^j$ for all $t > -\rho^2$. Let us extend $u^j$ to a smooth function on $Q_{\rho} := B_\rho \times(-\rho^2, 0]$ such that 
\[M_t^j \cap Z_\rho^j = \graph u^j (\cdot, t)\]
for each $t \in (-\rho^2, 0]$. 

Since $|A^j | \leq K$ for all $t \leq 0$ the Ecker-Huisken interior estimates \cite{Eck-Huisk} imply we can pass to a subsequence so the functions $u^j$ converge to a limit $u'$ in $C^\infty(Q_\rho)$. The hypersurfaces $N_t := \graph u'(\cdot,t)$ have nonnegative second fundamental form and move by mean curvature flow. Moreover, $H(y_j,0) \to 0$ implies the mean curvature of $N_0$ vanishes at the origin. Thus the strong maximum principle implies $N_0$ is a piece of a hyperplane. Recalling $y_j \to 0$ and $\nu^j(y_j,0) \to e_{n+1}$ we observe $N_0 = B_\rho$. 

Thus $\mathcal K$ contains $B_\rho$. Moreover, for each $z \in B_\rho$ there is a sequence $z_j \in M_0^j$ such that 
\[z_j \to 0, \qquad \nu^j(z_j,0) \to e_{n+1}, \qquad H^j(z_j,0) \to 0.\]
Therefore, repeating the above argument at finitely many points near $\partial B_\rho$ we conclude $B_{3\rho/2} \subset \mathcal K$. Continuing in this way we see, for any $R<\infty$, $B_R \subset \mathcal K$. That is, $\{x_{n+1} = 0\} \subset \mathcal K$. 

To finish, we show $\{x_{n+1} = 0\} \subset \mathcal K$ is incompatible with the Type I property. Lemma \ref{lem:type-I_envelope} tells us $\Omega^j_0 $ is a subset of $D(a_j\Omega_0, K)$, so passing to the limit we obtain 
\[\mathcal K \subset D(\mathcal T_\infty \Omega_0, K).\]
In particular, $D(\mathcal T_\infty \Omega_0, K)$ contains $\{x_{n+1} = 0\}$. Therefore, since $D(\mathcal T_\infty \Omega_0, K)$ is convex, it is the product of $\{x_{n+1} = 0\}$ with an interval. On the other hand $\mathcal T_\infty \Omega_0$ lies in a hyperplane by Proposition \ref{prop:asymptotic_cone}, and $0 \in \mathcal T_\infty \Omega_0$, so in fact
\[\mathcal T_\infty \Omega_0 = \{x_{n+1} = 0\},\]
hence $\Omega_t$ contains $\{x_{n+1} = 0\}$ for all $t <0$. By convexity, $\Omega_t$ is the product of $\{x_{n+1} = 0\}$ with an interval. That is, $M_t$ is a stationary hyperplane or pair of hyperplanes. However $M_t$ was assumed to have positive mean curvature, so this is a contradiction. 
\end{proof}

With the lower speed bound of Lemma \ref{lem:lower-speed_MCF} in hand we readily establish a lower bound for the displacement of $M_t$ from the origin. 

\begin{lemma}
\label{lem:displacement}
The ball $B(0, c_K \sqrt{-t})$ is contained in $\Omega_t$ for every $t \leq -3$. 
\end{lemma}

\begin{proof}
Recall we are assuming $M_0$ contains the origin. The function $f(t):=\dist(M_t, 0)$ is locally Lipschitz continuous, so its derivative exists at almost every time. At any time of differentiability, if $x \in M_t$ attains $\dist(M_t, 0)$ then
\[f'(t) \leq -H(x,t),\]
and $x$ is in $B(0,K\sqrt{-t})$ by Lemma \ref{lem:type-I_envelope}, so Lemma \ref{lem:lower-speed_MCF} implies $\sqrt{1 -t} \, H(x,t) \geq c_K $. We thus have
\[f'(t) \leq -\frac{c_K}{\sqrt{1 -t}},\]
which integrates to $f(t) \geq 2c_K(\sqrt{1-t}-1)$. From this the claim follows directly. 
\end{proof}

To prove the next result we extract a blow-down limit of $M_t$ which is smooth and convex. It should be emphasised that this is not possible in general - there are convex ancient solutions of mean curvature flow which blow down to a multiplicity-two hyperplane \cite{Wang}, \cite{BLT_pancake} and \cite{BLT_atomic}. In our setting this behaviour is ruled out by Lemma \ref{lem:lower-speed_MCF}. 

\begin{lemma}
\label{lem:split}
$\mathcal T_\infty \Omega_0$ is an affine subspace of $\mathbb{R}^{n+1}$.
\end{lemma}
\begin{proof}
Fix a positive sequence $a_j \to 0$ and define $\{M_t^j := a_j M_{a_j^{-2} t}\}_{t\in(-\infty,0]}$. Given $r>0$, \eqref{eq:Type I MCF} implies the second fundamental form of $M_t^j$ satisfies $|A^j| \leq K r^{-1} $ for all $t \leq -r^2$. In addition, Lemma \ref{lem:type-I_envelope} implies
\begin{equation}
\label{eq:blow-down_seq_containment}
M_t^j \subset D(a_j \Omega_0, K\sqrt{-t})
\end{equation}
for all $t\leq 0$. Finally, for large $j$ Lemma \ref{lem:displacement} implies $B(0, c_K \sqrt{-t})\subset\Omega_t^j$ for all $t \leq -r^2$. Appealing to the Ecker-Huisken interior estimates in \cite{Eck-Huisk} and the Arz\`{e}la-Ascoli theorem we find $\{M_t^j\}_{t \in (-\infty,0]}$ subconverges in $C^\infty_{\loc}(\mathbb{R}^{n+1} \times (-\infty, - r^2])$. Since $\Omega_t^j$ is convex this can be proven, for example, by expressing $M_t^j$ locally as a radial graph over $\partial B(0,c_K r)$ for $t \leq -r^2$. 

Since $r>0$ was arbitrary, by a diagonal argument there is a further subsequence of $\{M_t^j\}_{t\in(-\infty,0]}$ that converges in $C^\infty_{\loc}(\mathbb{R}^{n+1} \times (-\infty, 0))$. Denote the limit $\{M_t'=\partial \Omega_t'\}_{t\in(-\infty,0)}$, and define 
\[M_0' := \cap_{t<0} \, \Omega_t'.\]
Observe $M_0'$ is convex, since it is an intersection of convex sets. Moreover, since
\[\mathcal T_\infty \Omega_0 \subset a_j \Omega_{a_j^{-2} t}\]
for all $t <0$ we have $\mathcal T_\infty\Omega_0 \subset \Omega_t'$ for all $t <0$, hence $\mathcal T_\infty \Omega_0 \subset M_0'$. Using the Type I condition we also get the reverse inclusion: \eqref{eq:blow-down_seq_containment} implies 
\[\Omega_t' \subset D(\mathcal T_\infty \Omega_0, K\sqrt{-t})\]
for all $t<0$, hence $M_0 ' \subset \mathcal T_\infty \Omega_0$. Thus $\mathcal T_\infty \Omega_0$ and $M_0'$ coincide. 

Using Proposition \ref{prop:asymptotic_cone} we conclude $M_0'$ lies in a hyperplane. On the other hand,
\[B(0, c_K \sqrt{-t})\subset \Omega_t'\]
for all $t < 0$ by Lemma \ref{lem:displacement}. Combining these two observations we see that the hypersurfaces $M_t'$ sweep out all of space, in the following sense:
\[\mathbb{R}^{n+1} = \cl( \cup_{t < 0} \, M_t').\]
Solutions with this property are called entire, since their arrival time is an entire function. 

Since $\{M_t'\}_{t\in (-\infty,0)}$ is entire and convex, and $M_0'$ is non-empty, a result of Xu-Jia Wang \cite{Wang}[Lemma 2.9] asserts that $M_0'$ is an affine subspace of $\mathbb{R}^{n+1}$, possibly of dimension zero. This proves the claim, as we showed above $\mathcal T_\infty \Omega_0 = M_0'$.
\end{proof}

We finally conclude $\{M_t\}_{t\in(-\infty,0]}$ is spherical or cylindrical by a dimension-reduction argument. 

\begin{proof}[Proof of Theorem \ref{thm:main_MCF}]
By Lemma \ref{lem:split} we know $\mathcal T_\infty \Omega_0$ coincides with an affine subspace $\Sigma \subset \mathbb{R}^{n+1}$. Let $m$ denote the dimension of $\Sigma$. If $m=0$ we conclude $\Omega_0$ is compact. This in turn implies $M_t$ is compact, but every compact convex Type I ancient solution of mean curvature flow in $\mathbb{R}^{n+1}$, $n\geq 2$, is a family of homothetically shrinking spheres by work of Huisken-Sinestrari \cite{Huisk-Sin15}. 

Suppose then $m \geq 1$. In this case, since $\Sigma \subset \Omega_t$ for all $t <0$ and $\Omega_t$ is convex we have the splitting  
\[M_t = \Sigma \times M^\perp_t, \qquad M^\perp_t := M_t \cap \Sigma^\perp.\] 
It follows that $M_t^\perp$ is a Type I convex ancient solution of mean curvature flow in $\Sigma^\perp \cong \mathbb{R}^{n-m+1}$. Let $\Omega_t^\perp$ denote $\Omega_t \cap \Sigma^\perp$, so that $M_t^\perp = \partial \Omega_t^\perp$. Since 
\[\mathcal T_\infty \Omega_0^\perp = \mathcal T_\infty \Omega_0 \cap \Sigma^\perp = \Sigma \cap \Sigma^\perp = \{0\},\]
$M_t^\perp$ is compact. 

If $n-m=1$ we now appeal to the classification of compact convex ancient solutions to curve-shortening flow by Daskalopoulos-Hamilton-Sesum \cite{Dask-Ham-Ses} to conclude $\{M_t^\perp\}_{t\in(-\infty,0]}$ is a family of shrinking circles in $\mathbb{R}^2$. If instead $n-m \geq 2$, we again appeal to \cite{Huisk-Sin15} to conclude $\{M^\perp_t\}_{t \in (-\infty,0]}$ is a family of shrinking spheres. In either case $\{M_t\}_{t\in(-\infty,0]}$ is a family of shrinking cylinders. 
\end{proof}

\section{Fully nonlinear flows}
\label{sec:FNL}

In this section we establish Theorem \ref{thm:main_FNL}. The proof follows similar lines to Theorem \ref{thm:main_MCF}, but various new difficulties arise. In particular, Xu-Jia Wang's Lemma 2.9 in \cite{Wang} does not seem to generalise easily to flows by speeds other than the mean curvature. Therefore, we argue differently. We first prove Theorem \ref{thm:main_FNL} under the additional assumption that the solution is entire. This argument is partially inspired by the proof of Lemma 2.9 in \cite{Wang}, but is more difficult, notably requiring the Krylov-Safonov estimate, a characterisation of shrinking cylinder solutions proven using translating barriers, and uniqueness results for shrinking spheres from \cite{Lang-Lyn20}. We then proceed with the proof of Theorem \ref{thm:main_FNL} in general, using the entire case as an ingredient in a dimension-reduction argument.

\subsection{The entire case}\label{subsec:FNL_entire} A convex ancient $\gamma$-flow $\{M_t=\partial \Omega_t\}_{t\in(-\infty,0)}$ is entire if it sweeps out all of space, in the sense that 
\[\mathbb{R}^{n+1} = \cl(\cup_{t < 0} \, M_t).\]
Recall we say $\gamma$ admits a splitting theorem if every uniformly parabolic convex ancient $\gamma$-flow is either strictly convex or splits as a product $M_t = \mathbb{R}^m \times M_t^\perp$, where $M_t^\perp$ is strictly convex in $\mathbb{R}^{n-m+1}$. As a step towards proving Theorem \ref{thm:main_FNL} we establish:

\begin{proposition}
\label{thm:Wang_FNL}
Let $\gamma \in C^\infty(\Gamma)$ be an admissible speed which is either convex, or concave and inverse-concave, and admits a splitting theorem. Let $\{M_t = \partial \Omega_t\}_{t\in(-\infty,0)}$ be a uniformly parabolic convex entire $\gamma$-flow such that $|A| \leq K/\sqrt{-t}$ for all $t <0$. Then $\{M_t\}_{t\in(-\infty,0)}$ is a family of homothetically shrinking spheres or cylinders. 
\end{proposition}

We begin with two technical lemmas which are useful for proving lower speed bounds. For each $r>0$ we define a parabolic cylinder
\[Q_r:=\{ (x,t) \in \mathbb{R}^n \times (-\infty, 0] : |x|<r, \; -r^2 < t \leq 0\}.\]
We denote by $P^2(Q_r)$ the space of functions defined in $Q_r$ for which
\[\|u\|_{P^{2}(Q_r)} :={} \sup_{Q_r}  \;|u| + |u_t| + |Du| + |D^2 u| \]
is finite. In addition, we write $P^{2,\alpha}(Q_r)$ for the parabolic H\"{o}lder space with norm
\begin{align*}
\|u\|_{P^{2,\alpha}(Q_r)} :={} \|u\|_{P^{2}(Q_r)} &+ \sup_{(x,t)\not=(y,s) \in Q_r}  \frac{|u_t(x,t) - u_t(y,s)|+ |D^2 u(x,t) - D^2u(y,s)| }{|x - y|^\alpha + |t - s|^\frac{\alpha}{2}}\\
&+\sup_{(x,t)\not=(x,s) \in Q_r} \frac{|D u(x,t) - D u(x,s)|}{|t - s|^\frac{1+\alpha}{2}}.
\end{align*}

Recall a $\gamma$-flow is uniformly parabolic if there is a positive lower bound for $\dist(\lambda/G, \partial \Gamma)$ valid over the whole solution, where $\gamma \in C^\infty(\Gamma)$. Equivalently, the principal curvatures of $M_t$ lie in a closed symmetric convex cone $\Gamma_0$ satisfying 
\[\Gamma_0 \cap \{z \in \Gamma : \gamma(z) =1\} \subset \Gamma.\]
When this inclusion holds we write $\Gamma_0 \Subset \Gamma$. 

\begin{lemma}
\label{lem:Harnack}
Fix a concave or convex admissible speed $\gamma \in C^\infty(\Gamma)$. Let $\{M_t^j\}_{t \in (-r^2, 0]}$ be a sequence of uniformly parabolic $\gamma$-flows with principal curvatures in a fixed cone $\Gamma_0 \Subset \Gamma$, and suppose there is a corresponding sequence of smooth functions $u^j : Q_r \to \mathbb{R}$ such that $M_t^j = \graph u^j(\cdot, t)$ and $\|u^j\|_{P^2(Q_r)} \leq K$. Finally, suppose $G^j(x_j, 0) \to 0$ for some sequence $x_j \in M_0^j$ such that $x_j \to 0$ and $\nu^j(x_j,0) \to e_{n+1}$. For every $\theta<1$, up to a subsequence, $u^j \to 0$ in $P^2(Q_{\theta r})$. 
\end{lemma}
\begin{proof}
To ease notation, fix $j$ and write $u$ for $u^j$. With respect to the parameterisation induced by $u$ the normal to $M_t^j$ and its Weingarten map are given by
\[\nu = \frac{1}{\sqrt{1+|Du|^2}} (-Du,1), \qquad   W^i_j = -g^{ik} \frac{D_k D_j u}{\sqrt{1+|Du|^2}},\]
where
\[g^{ik}:=\delta^{ik} - \frac{D_i u D_k u}{1+|Du|^2} .\]
As in \cite{Urbas} we define 
\[P^i_j := \delta^i_j - \frac{D_iu D_ju}{\sqrt{1+|Du|^2}(1+\sqrt{1+|Du|^2})}, \qquad \hat W_{ij} := -P^k_i \frac{D_k D_l u}{\sqrt{1+|Du|^2}} P^l_j,\]
so that $\hat W$ is symmetric but has the same eigenvalues as $W$. Since the hypersurfaces $M_t^j$ form a $\gamma$-flow and 
\[G(x + u(x,t)e_{n+1},t) = \gamma(\lambda(x,t)) = \gamma(\hat W(x,t))\]
we have
\[\frac{u_t}{\sqrt{1+|Du|^2}} = - \gamma(\hat W).\]
Differentiating this equation we find $v:=-u_t$ solves an equation of the form 
\[v_t = a^{ij} D_i D_j v + b^i D_i v,\]
where $a^{ij} := P^i_k \dot \gamma^{kl}(\hat W) P^j_l$, and $b^i$ is some expression involving $\dot \gamma$ and first and second spatial derivatives of $u$ such that $|b^j| \leq C(K)$. We also have 
\[C^{-1} |\xi|^2 \leq a^{ij} \xi_i \xi_j \leq C |\xi|^2\]
for some $C = C(n,\gamma, \Gamma_0, K)$, so by the Krylov-Safonov estimate \cite{Kryl-Saf}, for every $\theta <1$ and $\tau < 0$, 
\[\sup_{B(0,\theta r)} v(\cdot, \tau) \leq C \inf_{B(0,\theta r)} v(\cdot, 0),\]
where $C = C(n,\gamma, \Gamma_0, K, \theta, \tau)$. Observing that $v = \sqrt{1+|Du|^2}\, \gamma(\hat W)$, we obtain 
\[\sup_{B(0,\theta r)} \gamma(\hat W(\cdot, \tau)) \leq C \inf_{B(0,\theta r)}\gamma(\hat W(\cdot, 0)),\]
where $C = C(n,\gamma, \Gamma_0, K, \theta, \tau)$.

Since $j$ was arbitrary this inequality holds with the same constant for each $M_t^j$. In particular, since $G^j(x_j,0) \to 0$, for each $\theta <1$ we conclude
\begin{equation}
\label{eq:Hess to zero}
\sup_{B(0, \theta r)} |D^2u^j|(\cdot, t) \to 0
\end{equation}
as $j \to \infty$ for every $t \in (-\theta^2 r^2, 0)$. Consider now a fixed $\theta<1$. Since $\gamma$ is concave or convex, Krylov's estimate \cite{Kryl82} implies $\|u^j\|_{P^{2,\alpha}(Q_{\theta r})}$ can be bounded purely in terms of $n$, $K$ and $\theta$ for some $\alpha \in (0,1)$ depending only on $n$ and $K$. Passing to a subsequence, we may therefore assume $u^j$ converges in $P^2(Q_{\theta r})$. As a consequence of \eqref{eq:Hess to zero} the limit is linear in space and constant in time, hence zero since $x_j \to 0$ and $\nu^j(x_j, 0) \to 0$. 
\end{proof}

\begin{lemma}
\label{lem:tech}
Fix a concave or convex admissible speed $\gamma \in C^\infty(\Gamma)$. Let $\{M_t^j\}_{t\in(-\infty,0]}$ be a sequence of convex $\gamma$-flows with principal curvatures in $\Gamma_0 \Subset \Gamma$. Suppose $|A^j|\leq K$ for some uniform $K$ and that $G^j(x_j,0) \to 0$ along a bounded sequence $x_j$. Then there is a hyperplane $\Sigma \subset \mathbb{R}^{n+1}$ and a closed interval $I$ such that, for each $t \leq 0$, the sequence $\Omega_t^j$ subconverges to $\Sigma \times I$ in the Hausdorff topology.
\end{lemma}
\begin{proof}
Employing a diagonal argument we find there is a subsequence in $j$ such that $\Omega_t^j$ converges to a closed convex limit $\mathcal K_t$ in the Hausdorff topology for every rational $t \leq 0$. The pointwise speed of $M_t^j$ is bounded in terms of $K$ and $\Gamma_0$, so for every compact $X \subset \mathbb{R}^{n+1}$, the distance between $\mathcal K_t \cap X$ and $\mathcal K_{t+h} \cap X$ is $O(h)$. Since the sets $\Omega_t^j$ move inward, using the triangle inequality we conclude $\Omega_t^j$ converges to a closed convex limit $\mathcal K_t$ in the Hausdorff topology for every $t \leq 0$. 

The sequence $x_j$ is bounded so we may pass to a subsequence such that, up to a rigid motion, $x_j \to 0$ and $\nu^j(x_j,0) \to e_{n+1}$. Let $U^j$ denote the largest connected neighbourhood of $x_j$ in $M_0^j$ in which $\nu^j \cdot e_{n+1} > \tfrac{1}{2}$ and denote by $P$ the orthogonal projection from $\mathbb{R}^{n+1}$ to $\{x_{n+1} = 0\}$. We have $|A^j| \leq K$, so up to a further subsequence there is an $r = r(K)$ such that 
\[B_r \subset P(U^j),\]
where $B_r := \{x_{n+1} = 0\} \cap B(0,r)$. By the inverse function theorem, $P^{-1}(B_r) \cap U^j$ is the graph of a smooth function $u^j : B_r \to \mathbb{R}$. For each $\rho < r$ define
\[Z^j_\rho := P^{-1}(B_\rho) \cap \{x_{n+1}\geq u^j(x_1, \dots, x_n)\}.\]
Since $\Omega_t^j$ is convex, $M_t^j \cap Z^j_\rho$ is a graph over $B_\rho$ for all $t \leq 0$. In fact we may choose $\rho$ depending only on $K$ so that $\nu^j(x,t) \cdot e_{n+1} \geq \tfrac{1}{2}$ for all $x \in M_t^j \cap Z_\rho^j$ and $t > -\rho^2$. In particular, $u^j$ can be extended to a smooth function $u^j : Q_\rho \to \mathbb{R}$ such that 
\[M_t^j \cap Z_\rho^j = \graph u^j(\cdot, t)\]
for all $t \in (-\rho^2, 0]$, where $Q_\rho := B_\rho \times (-\rho^2, 0]$. 

The bounds $|A^j|\leq K$ and $\nu^j \cdot e_{n+1} \geq \tfrac{1}{2}$ imply $\|u^j\|_{P^2(Q_\rho)}$ is bounded independently of $j$, so since $G^j(x_j, 0) \to 0$, appealing to Lemma \ref{lem:Harnack} we conclude $u^j$ subconverges to 0 in $P^2(Q_{\theta \rho})$ for all $\theta <1$. Hence $B_{\rho} \subset \partial \mathcal K_t$ for all $t > -\rho^2$, and for any $y \in B_\rho$ there is a sequence $y^j \in M_0^j$ such that
\[y_j \to 0, \qquad G^j(y_j, 0) \to 0, \qquad \nu^j(y_j, 0) \to e_{n+1},\]
so we may apply Lemma \ref{lem:Harnack} again at finitely many points near $\partial B_\rho$ to conclude $B_{3\rho/2} \subset \partial \mathcal K_t$ for all $t > -\rho^2$. Given any $R< \infty$, after repeating this procedure finitely many times we find
\[B_R \subset \partial \mathcal K_t\]
for all $t > -\rho^2$. We may also iterate backwards in time in exactly the same way to conclude, for all $R < \infty$, $B_R \subset \partial \mathcal K_t$ for all $t \geq - R^2$. In other words, $\{x_{n+1} = 0\} \subset \partial \mathcal K_t$ for all $t \leq 0$. 

Define $\Sigma = \{x_{n+1} = 0\}$. Since $\mathcal K_0$ is convex and contains $\Sigma$, $\mathcal K_0 = \Sigma \times I$ where $I$ is a closed interval. If $I = [0,\infty)$, which is to say $\mathcal K_0$ is a halfspace, then since $ \mathcal K_0 \subset \mathcal K_t $ and $\Sigma \subset \partial \mathcal K_t$ we conclude $\mathcal K_t=\Sigma \times I$ for all $t \leq 0$. If $I = [0,a]$ (where we allow $a =0$) then for any $R\gg a$, when $j$ is large, $M_0^j \cap B(0,R)$ consists of two connected, approximately planar components. One contains $x_j$, and in the other we can find a bounded sequence $z_j$ such that $G^j(z_j, 0) \to 0$. Repeating the above argument we find $\Sigma \times \{a\}$ is a subset of $\partial \mathcal K_t$ for all $t \leq 0$. Then since $\mathcal K_t$ is convex, $\mathcal K_t = \Sigma \times [0,a]$. 
\end{proof}

Next we construct a family of translating barriers. For each $0 \leq m \leq n-1$ let us write 
\[\Cyl_m:=\{(\underbrace{0,\dots, 0}_{m}, r, \dots, r) \in \mathbb{R}^n : r >0\}.\]
Given an admissible speed $\gamma \in C^\infty(\Gamma)$ we define $\bar m(\Gamma) := \max\{ 0\leq m\leq n-1 : \Cyl_m \subset \Gamma\}$, and set
\[c_m := \gamma(\underbrace{0,\dots,0}_{m},1,\dots,1)^{-1}\]
for each $0\leq m \leq \bar m(\Gamma)$. That is, $c_m^{-1}$ is the value taken by $G$ on a round unit cylinder $\mathbb{R}^{m}\times S^{n-m}$. For each $0\leq m\leq \bar m(\Gamma)$ and time $t \leq 0$ we define
\[\mathcal C^m_t := \{x \in \mathbb{R}^{n+1} : x_1^2 + \dots x_{n-m+1}^2 \leq -2 c_{m}^{-1} t\},\]
so that the family of cylinders $\{\partial \mathcal C_t^m\}_{t\in(-\infty,0)}$ is a $\gamma$-flow. For $m > \bar m(\Gamma)$ there is no cylindrical $\gamma$-flow modeled on $\mathbb{R}^m \times S^{n-m}$, since $\gamma$ isn't even defined on $\Cyl_m$. For example, the two-harmonic mean curvature flow, for which 
\[G = \Big(\sum_{i<j} \frac{1}{\lambda_i + \lambda_j}\Big)^{-1},\]
admits a smooth shrinking cylinder solution with at most one Euclidean factor. 

\begin{lemma}
\label{lem:bowl}
Fix an admissible speed $\gamma \in C^\infty(\Gamma)$. For each $1 \leq m \leq \bar m(\Gamma)$, define
\[\mathcal B^m_t := \{x \in \mathbb{R}^{n+1}: x_{n+1} \geq \tfrac{c_{m}}{2}(x_1^2 + \dots + x_{n-m+1}^2) + t\}.\]
Write $\lambda(x,t)$ for the principal curvatures of $\partial \mathcal B_t$ at $x$. As $t$ increases the hypersurfaces $\partial \mathcal B_t^m$ translate in the direction of $e_{n+1}$ with pointwise normal speed not exceeding $G = \gamma(\lambda)$.
\end{lemma}
\begin{proof}
Fix $1 \leq m \leq \bar m(\Gamma)$. Given $x \in \mathbb{R}^{n}$, write $x'$ for the point $(x_1, \dots, x_{n-m+1},0,\dots,0)$. Then the family of embeddings $F:\mathbb{R}^n \times (-\infty,0] \to \mathbb{R}$ defined by
\[F(x,t) = (x, \tfrac{c_{m}}{2}|x'|^2 + t)\]
parameterises $\{\partial \mathcal B_t^m\}_{t\in(-\infty,0]}$. We choose the downward normal to $\partial \mathcal B_t^m$, 
\[\nu(x,t) = \frac{(c_m x', -1)}{\sqrt{1+c_{m}^2|x'|^2}},\]
so the principal curvatures of $\partial \mathcal B_t^m$ at $F(x,t)$ are
\[\lambda(x,t) = \frac{c_{m}}{\sqrt{1+c_{m}^2|x'|^2}} \bigg(0,\dots,0, 1 - \frac{c_{m}^2 |x'|^2}{1+c_{m}^2|x'|^2}, 1, \dots, 1\bigg),\]
where zero occurs with multiplicity $m-1$. Since $\gamma$ is increasing in each argument, we have
\[G(x,t) \geq \frac{c_{m}}{\sqrt{1+c_{m}^2 |x'|^2}} \gamma(\underbrace{0,\dots,0}_{m},1,\dots,1) =  \frac{1}{\sqrt{1+c_{m}^2|x'|^2}},\]
hence 
\[\partial_t F(x,t) \cdot \nu(x,t) = -\frac{1}{\sqrt{1+c_{m}^2|x'|^2}} \geq - G(x,t).\]
\end{proof}

Observe that as $a \to 0$ the rescaled domains $\{ a \mathcal B^m_{a^{-2} t}\}_{t \in (-\infty,0)}$ converge to $\{\mathcal C^{m}_t\}_{t\in(-\infty,0)}$ in $C_{\loc}^\infty(\mathbb{R}^{n+1} \times (-\infty,0))$. Exploiting this fact we prove:

\begin{proposition}
\label{prop:cyl_char_FNL}
Fix an admissible speed $\gamma \in C^\infty(\Gamma)$. Let $\{M_t = \partial \Omega_t\}_{t\in (-\infty,0)}$ be a convex $\gamma$-flow such that 
$M_t \subset \mathcal C_t^m$
for all $t <0$, and suppose $\cap_{t<0} \, \Omega_t$ is nonempty. Then $M_t = \partial \mathcal C_t^m$ for all $t<0$. 
\end{proposition}
\begin{proof}
If $m =0$ the claim is a simple consequence of the avoidance principle, so suppose $m \geq 1$. 
Observe that $M_0:= \cap_{t < 0} \, \Omega_t $ is convex and lies in the $m$-dimensional linear subspace
\[\mathcal C_0^m = \{x_1= \dots = x_{n-m+1}=0\}.\]
Let $\ell$ denote the dimension of the largest affine subspace of $\mathbb{R}^{n+1}$ contained in $M_0$. The key step is to prove the claim in case $\ell = 0$, since the remaining cases then follow by a dimension-reduction argument.

Suppose then $\ell =0$. In this case we claim $M_0$ is compact, so suppose for a contradiction $M_0$ is noncompact. Then there is a unit vector $e \in M_0$ with the following properties:
\begin{itemize}
\item For each $r \geq 0$, $r e \in M_0$.
\item The quantity $r_0:=\inf_{x \in M_0} x \cdot e$ is finite.
\item For each $r \geq r_0$ the set $M_0 \cap \{x\cdot e = r\}$ is compact.
\end{itemize}
Applying a rigid motion in the ambient space (leaving $\mathcal C_0^m$ fixed), we can arrange that $e = e_{n+1}$ and $r_0 < 0$. In particular $M_0$ fails to lie in the halfspace $\{x_{n+1} \geq 0\}$. Let $E$ denote the line through the origin parallel to $e$. 

Fix $T <0$. We have arranged that $M_T \cap \{x_{n+1} \leq 0\}$ is compact, so since $M_T \subset \mathcal C_T^m$,
\begin{equation}
\label{eq:cyl_dist_pos}
\dist(M_T \cap \{x_{n+1} \leq 0\}, \,\partial \mathcal C_T^m) >0. 
\end{equation}
Indeed, if this fails the strong maximum principle implies $M_T = \partial \mathcal C_T^m$, but this is incompatible with $\ell =0$. 

We now make use of Lemma \ref{lem:bowl}. Taking $a$ sufficiently small, $a \mathcal B^{m}_{a^{-2} T}$ can be made arbitrarily close to $ \mathcal C_T^m$ in any compact subset of $\mathbb{R}^{n+1}$. In particular, by \eqref{eq:cyl_dist_pos} we can choose $a$ so small (depending on $T$) that 
\[M_T \cap \{x_{n+1} \leq 0\} \subset a \mathcal B^{m}_{a^{-2}T}.\]
On the other hand, decreasing $a$ further if necessary, we can arrange that 
\[M_T \cap \{x_{n+1} \geq 0\} \subset \mathcal C_{T}^m \cap \{x_{n+1} \geq 0\} \subset a \mathcal B_{a^{-2}T}^{m} ,\]
hence $M_T \subset a\mathcal B^m_{a^{-2}T}$. Finally, we observe that for every $t < 0$ there is an $R<\infty$ for which
\[ \mathcal C^m_t \cap \{x_{n+1}\geq R\} \subset a \mathcal B_{a^{-2}t}^{m},\]
so Lemma \ref{lem:bowl} and the avoidance principle imply $M_t \subset a \mathcal B^{m}_{a^{-2}t}$ for all $t \in [T,0)$. 

With this we have reached a contradiction, since $a \mathcal B_0^{m}$ lies in $\{x_{n+1} \geq 0\}$, but $M_0$ does not. Hence our original assumption was false; that is, $M_0$ is compact. It follows that $M_t$ is compact for each $t <0$. We are assuming $M_t$ does not go extinct until $t =0$, so the avoidance principle implies $M_t$ and $\partial \mathcal C_t^m$ coincide for all $t <0$. Since $M_t$ is compact, $m=0$ and $\{M_t\}_{t\in(-\infty,0)}$ is a family of homothetically shrinking spheres. This concludes the case $\ell = 0$. 

Now suppose $\ell\geq 1$ and let $\Sigma$ be an $\ell$-dimensional affine subspace of $\mathbb{R}^{n+1}$ sitting inside $M_0$. Then since $\Sigma \subset \Omega_t$ for all $t<0$, by convexity, we have the splitting $M_t = \Sigma \times M_t^\perp$ where $M_t^\perp := M_t \cap \Sigma^\perp$. The family $\{M_t^\perp\}_{t\in(-\infty,0)}$ is an $(n -\ell)$-dimensional $\gamma^{(\ell)} $-flow in $\Sigma^\perp \cong \mathbb{R}^{n-\ell+1}$, where 
\[\gamma^{(\ell)} (z) := \gamma(0, \dots, 0, z_1, \dots, z_{n-\ell}).\]
Moreover, by maximality of $\ell$ the only affine subspaces in 
\[M_0^\perp := \cap_{t <0} \, \Omega_t^\perp = M_0 \cap \Sigma^\perp\]
have dimension zero. Therefore, since $M_t^\perp \subset \mathcal C^m_t \cap \Sigma^\perp$ for each $t<0$ and $\{\mathcal C_t^m \cap \Sigma^\perp\}_{t\in(-\infty,0)}$ is a $\gamma^{(\ell)}$-flow in $\Sigma^\perp$, exactly as above, $M_t^\perp$ is a family of homothetically shrinking spheres. Consequently, $\{M_t\}_{t\in(-\infty,0)}$ is a family of homothetically shrinking cylinders. Since $M_t$ does not go extinct until $t =0$ and lies inside $\mathcal C^m_t$ for all $t<0$, the avoidance principle implies $M_t = \partial \mathcal C_t^m$ for all $t <0$. 
\end{proof}

The proof of Proposition \ref{thm:Wang_FNL} is by induction, with the following lemma comprising the base step. Recall the notation $\Gamma_+ := \{\lambda \in \mathbb{R}^n : \lambda_i > 0\}$. 

\begin{lemma}
\label{lem:Wang_FNL_min}
Fix a concave or convex admissible speed $\gamma \in C^\infty(\Gamma)$. Let $\{M_t=\partial \Omega_t\}_{t\in(-\infty,0)}$ be a uniformly parabolic convex entire $\gamma$-flow and set $m = \bar m(\Gamma)$. If $\cap_{t<0}\, \Omega_t$ contains an $m$-dimensional affine subspace of $\mathbb{R}^{n+1}$, $\{M_t\}_{t\in(-\infty,0)}$ is the product of $\mathbb{R}^{m}$ with a family of homothetically shrinking spheres.  
\end{lemma}

\begin{proof}
Suppose $M_0 := \cap_{t < 0}\, \Omega_t$ contains an affine subspace of dimension $m$, which we denote by $\Sigma$. Since $M_0 \subset \Omega_t$ for all $t < 0$ and $\Omega_t$ is convex we have the splitting $M_t = \Sigma \times M^\perp_t$, where each $M^\perp_t$ is a smooth hypersurface in $\Sigma^\perp \cong \mathbb{R}^{\ell+1}$, $\ell:=n-m$. In particular the principal curvatures of $M_t$ lie in the facet of $\partial \Gamma_+$ defined by $\{\lambda_1 = \dots = \lambda_m =0\}$. Let us write $\Gamma^\ell_+ := \{z \in \mathbb{R}^\ell : z_i > 0\}$ and denote by $\psi: \Gamma^\ell_+ \to \Gamma$ the map $\psi(z) := (0, \dots, 0, z)$. Then the family $\{M_t^\perp\}_{t \in (-\infty,0)}$ is a convex $\gamma_\ell$-flow in $\mathbb{R}^{\ell+1}$, where $\gamma_\ell \in C^\infty(\Gamma_+^\ell)$ is the composition $\gamma \circ \psi$. 

In case $\ell =1$, up to a rescaling of time, $M_t^\perp$ is an entire convex ancient solution of curve-shortening flow in $\mathbb{R}^2$. Hence $\{M_t^\perp\}_{t\in(-\infty,0)}$ is a family of shrinking circles by \cite{BTL20}, and the claim is proven. 

Assume then $\ell \geq 2$. Let us write $\lambda_1^\perp \leq \dots \leq \lambda_\ell^\perp$ and $H^\perp$ for the principal curvatures and mean curvature of $M_t^\perp$. By uniform parabolicity there is a closed symmetric convex cone $\Gamma_0$ satisfying $\Gamma_0 \cap \{z \in \mathbb{R}^n : \tr(z) = 1\} \subset \Gamma$ such that $\lambda$ takes values in $\Gamma_0$. Hence $\lambda^\perp$ takes values in $\Gamma_0^\ell := \psi^{-1}(\Gamma_0)$. 

We claim $Z:=\Gamma_0^\ell \cap\{z \in \mathbb{R}^\ell: \tr(z) =1\}$ is a compact subset of $\Gamma_+^\ell$. If not, there is a $z \in Z$ such that $z_1 = 0$. Equivalently, there is a $z \in \Gamma_0$ such that $\tr(z) = 1$ and $z_1 = \dots = z_{m+1} = 0$. Since $\Gamma_0$ is a symmetric convex cone, we can cyclically permute the entries $z_i$, $m+2 \leq i \leq n$, and sum up to find 
\[\bar z = (\underbrace{0, \dots, 0}_{m+1}, 1, \dots ,1)\]
is an element of $\Gamma_0$. However, $\bar z \in \Cyl_{m+1}$. This is a contradiction since, by definition, $m=\bar m(\Gamma)$ is the maximal integer such that $\Cyl_m \subset \Gamma$. 

We thus conclude $\varepsilon := \min\{z_i: z \in Z\}$ is strictly positive, and consequently find 
\[ \lambda_1^\perp \geq \varepsilon H^\perp\]
on $M_t^\perp$ for all $t \leq 0$. In other words $M_t^\perp$ is uniformly convex, and by Hamilton's result \cite{Ham94}, $M_t^\perp$ is compact. In summary $\{M_t^\perp\}_{t\in(-\infty,0)}$ is a uniformly parabolic $\gamma_\ell$-flow which is also compact and uniformly convex. Theorem 1.7  in \cite{Lang-Lyn20} thus implies $\{M_t^\perp\}_{t\in(-\infty,0)}$ is a family of shrinking spheres.
\end{proof}

\begin{proof}[Proof of Proposition \ref{thm:Wang_FNL}]
Let $\{M_t\}_{t \in(-\infty,0)}$ be a uniformly parabolic convex entire $\gamma$-flow, and suppose there is a $K<\infty$ such that 
\[|A| \leq \frac{K}{\sqrt{-t}}\]
for all $t<0$. By uniform parabolicity, after increasing $K$ we may assume $G \leq K/\sqrt{-t}$ for all $t <0$. Integrating this speed bound (as in Lemma \ref{lem:type-I_envelope}) we find that $M_0 := \cap_{t < 0} \, \Omega_t$ is nonempty. Let $m$ be the dimension of the largest affine subspace contained in $M_0$. 

We proceed by induction on $m$. By Lemma \ref{lem:Wang_FNL_min}, if $m = \bar m(\Gamma)$ then $\{M_t\}_{t\in(-\infty,0)}$ is a family of spheres or cylinders.

Suppose then $m = k< \bar m(\Gamma)$, and that the claim has been proven for all solutions which satisfy the hypotheses of the proposition but contain an affine subspace of dimension $k+1$. Let $\Sigma$ be an affine subspace of dimension $k$ in $M_0$, and shift the solution if necessary so $0 \in \Sigma$. If $M_0 \cap \Sigma^\perp$ is compact then $M_t$ is a compact convex ancient solution of Type I, hence a sphere by \cite{Lang-Lyn20} (see also the argument in the proof of Theorem \ref{thm:main_FNL} below).

Now consider the case when $M_0 \cap \Sigma^\perp$ is noncompact. Then, since $M_0 \cap \Sigma^\perp$ is convex, it contains a ray $\{r e: r \geq 0\}$, $e \in \partial B(0,1)$. If $M_0\cap \Sigma^\perp$ contains the line $E:=\{r e: r \in \mathbb{R}\}$ then $M_0$ contains the span of $e$ and $\Sigma$, which has dimension $k+1$, so in this case the inductive hypothesis implies $M_t$ is cylindrical as required. Alternatively,
\begin{equation}
\label{eq:Wang_FNL_contra}
E \not \subset M_0.
\end{equation}
We complete the proof by showing this leads to a contradiction. 

Suppose \eqref{eq:Wang_FNL_contra} holds and form a sequence of shifted solutions $\{M_t^j:=M_t -j e\}_{t\in(-\infty,0)}$. Define also  $\Omega_t^j := \Omega_t - j e$. We claim $\{M_t^j\}_{t\in(-\infty,0)}$ subconverges in $C^\infty_{\loc}(\mathbb{R}^{n+1} \times (-\infty,0))$. Since $\{M_t^j\}_{t\in(-\infty,0)}$ is entire, for every $r>0$ there is a time $\tau<0$ such that 
\[B(0,r) \subset \Omega_{\tau} \subset \Omega_{\tau}^j,\]
hence $B(0,r) \subset \Omega_t^j$ for all $t \leq \tau$. We also know the second fundamental form of $M_t^j$ satisfies $|A^j| \leq K / \sqrt{ -\tau}$ for all $t \leq \tau$, and since $M_t^j$ reaches the origin as $t \to 0$ the avoidance principle guarantees $\dist(M_{\tau}^j, 0)$ is bounded uniformly in $j$. Finally, in every compact subset of $\mathbb{R}^{n+1}\times(-\infty,\tau]$ we have a lower bound for $G^j$ that is uniform in $j$; if not then Lemma \ref{lem:tech} implies $\{\Omega_t^j\}_{t\in(-\infty, \tau]}$ subconverges to a stationary slab or halfspace, but $\Omega_t \subset \Omega_t^j$ and $M_t$ is entire, so this is impossible.

Our uniform upper bound for $|A^j|$ ensures that, for $t \leq \tau$, each point in $M_t^j$ is contained in a neighbourhood of uniform size that can be expressed as the graph of a function, either over the tangent space, or over $\partial B(0,r)$. As in Lemma \ref{lem:Harnack}, Krylov's estimate implies H\"{o}lder bounds for the Hessian and time derivative of each such local height function, and with these in hand we can bootstrap using the Schauder estimates to conclude all of the higher covariant derivatives of $A^j$ are bounded independently of $j$ in compact subsets of $\mathbb{R}^{n+1} \times (-\infty, \tau]$. We note the bootstrapping step requires our lower bound on $G^j$, since all of the derivatives of $\gamma$ in $A$ of order at least two blow up as $A \to 0$. With these derivative estimates established we may pass to a subsequence in $j$ such that $\{M_t^j\}_{t \in(-\infty,0)}$ converges in $C^\infty_{\loc}(\mathbb{R}^{n+1} \times (-\infty,\tau])$. Since $r$ was arbitrary a diagonal argument shows that $\{M_t^j\}_{t\in(-\infty,0)}$ subconverges to a limit $\{M_t'= \partial \Omega_t'\}_{t\in (-\infty,0)}$ in $C^\infty_{\loc} (\mathbb{R}^{n+1}\times(-\infty,0))$.

Now, $\{M_t'\}_{t \in(-\infty,0)}$ is uniformly parabolic and convex, and a straightforward argument using the avoidance principle shows $\{M_t'\}_{t \in(-\infty,0)}$ is also entire. By construction, $\Omega_t'$ contains $E$ and $M_0$, so since $\Omega_t'$ is convex it contains the span of $e$ and $\Sigma$, which has dimension $k+1$. The inductive hypothesis thus implies $\{M_t'\}_{t\in(-\infty,0)}$ is a family of shrinking cylinders. On the other hand $\Omega_t \subset \cl(\Omega_t')$ for all $t<0$, so appealing to Proposition \ref{prop:cyl_char_FNL} we conclude that $M_t = M_t'$ for all $t <0$. In particular $\Omega_t$ contains $E$, contradicting \eqref{eq:Wang_FNL_contra}.
\end{proof}

\subsection{The general case} \label{subsec:FNL_TypeI}

Consider an admissible speed $\gamma\in C^\infty(\Gamma)$ which is convex, or concave and inverse-concave, and admits a splitting theorem. Let $\{M_t=\partial \Omega_t\}_{t\in(-\infty,0]}$ be a uniformly parabolic convex ancient $\gamma$-flow of Type I. Our goal is to show $M_t$ is spherical or cylindrical. Without loss of generality we may assume there is a $K<\infty$ such that
\[\sup_{t \leq 0} \,\sup_{M_t} \sqrt{-t} \, G \leq \frac{K}{2},\]
where $G:=\gamma(\lambda)$. Translating as necessary, we may assume $M_0$ contains the origin. 

Exactly as in Lemma \ref{lem:type-I_envelope}, integrating the Type I property yields an upper bound for the displacement of $M_t$.
\begin{lemma}
\label{lem:type-I_envelope_FNL}
For every $t\leq 0$ we have $M_t \subset D(\Omega_0 , K\sqrt{-t})$.
\end{lemma}

The next step is to establish a lower speed bound valid near the spacetime origin. 

\begin{lemma}
\label{lem:lower-speed_FNL}
For each $L>0$ there is a positive constant $c_L$ such that the inequality
\[\sqrt{1-t}\, G(x,t) \geq c_L \]
holds for every $t\leq 0$ and $x\in M_t \cap B(0,L\sqrt{-t} )$. 
\end{lemma}

\begin{proof}
We argue by contradiction. If the claim is false for some $L>0$, there is a sequence of times $t_j \to -\infty$ and points $x_j \in M_{t_j} \cap B(0, L\sqrt{-t_j})$ such that 
\[\sqrt{-t_j}\,G(x_j,t_j) \to 0.\]
Let us write $a_j = 1/\sqrt{-t_j}$ and form a sequence of rescaled solutions,
\[ \{M_t^j:=a_j M_{a_j^{-2} t}\}_{t\in(-\infty,-1]}.\] 
Lemma \ref{lem:type-I_envelope_FNL} implies the inclusion $M_{t}^j \subset D(a_j \Omega_0, K\sqrt{-t})$ and the Type I assumption and uniform parabolicity imply $A^j$ is bounded independently of $j$ for all $t \leq -1$. We have
\[G^j(a_j x_j, -1) = \sqrt{-t_j} \, G(x_j, t_j) \to 0,\]
so appealing to Lemma \ref{lem:tech} we find there is a hyperplane $\Sigma$ and an interval $I$ such that $\Omega_{-1}^{j}$ subconverges to $\Sigma \times I$ in the Hausdorff topology. In particular,
\[\Sigma \times I \subset D(\mathcal T_\infty \Omega_0, K).\]
This is only possible if $\mathcal T_\infty \Omega_0$ contains a translate of $\Sigma$, so by Proposition \ref{prop:asymptotic_cone} we conclude $\mathcal T_\infty \Omega_0$ is a hyperplane. Since $\Omega_t$ is convex and contains $\mathcal T_\infty \Omega_0$, $M_t$ is a stationary hyperplane or pair of hyperplanes, but this contradicts $G >0$. 
\end{proof}

With the lower speed bound in hand, the same argument used to prove Lemma \ref{lem:displacement_FNL} implies the following lower bound for displacement.

\begin{lemma}
\label{lem:displacement_FNL}
The ball $B(0, c_K \sqrt{-t})$ is contained in $\Omega_t$ for every $t \leq -3$. 
\end{lemma}

We now take a blow-down of $M_t$ and conclude:

\begin{lemma}
\label{lem:final_slice_FNL}
$\mathcal T_\infty \Omega_0$ is an affine subspace of $\mathbb{R}^{n+1}$. 
\end{lemma}
\begin{proof}
Given $r>0$ the Type I hypothesis implies the second fundamental form of $M_t^j$ satisfies 
\[|A^j| \leq C(n,\gamma,\dist(\lambda/G, \partial \Gamma),K) \,r^{-1} \]
for all $t \leq -r^2$. By Lemma \ref{lem:type-I_envelope_FNL} we know
\begin{equation}
\label{eq:blow-down_containment_FNL_seq}
M_t^j \subset D(a_j \Omega_0, K\sqrt{-t}),
\end{equation}
and as a consequence of Lemma \ref{lem:displacement_FNL} the ball $B(0, c_K r)$ is contained in $\Omega_t^j$ for all $t \leq -r^2$, at least when $j$ is large. Finally, by Lemma \ref{lem:lower-speed_FNL}, in every compact subset of $\mathbb{R}^{n+1} \times (-\infty,-r^2]$ we have a lower bound for $G^j$ that is uniform in $j$, so as above the Krylov and Schauder estimates imply $\{M_t^j\}_{t \in (-\infty,0]}$ subconverges in $C^\infty_{\loc}(\mathbb{R}^{n+1} \times (-\infty, - r^{2}])$. Since $r>0$ was arbitrary, a diagonal argument shows we have convergence in $C^\infty_{\loc}(\mathbb{R}^{n+1} \times (-\infty, 0))$; denote the limiting solution $\{M_t'=\partial \Omega_t'\}_{t\in(-\infty,0)}$. Then we have
\[M_t' \subset D(\mathcal T_\infty \Omega_0, K\sqrt{-t}), \qquad B(0, c_K \sqrt{-t}) \subset \Omega_t',\]
so $\{M_t'\}_{t\in(-\infty, 0)}$ is entire by Proposition \ref{prop:asymptotic_cone}. Moreover, from the first containment we see 
\[M_0':= \cap_{t<0} \, \Omega_t' \subset \mathcal T_\infty \Omega_0.\]
On the other hand
\[\mathcal T_{\infty} \Omega_0 \subset a_j \Omega_{0} \subset a_j \Omega_{a_j^{-2} t}\]
implies $\mathcal T_\infty \Omega_0 \subset M_0'$, so in fact $\mathcal T_\infty \Omega_0 = M_0'$. 

To conclude we observe $\{M_t'\}_{t\in(-\infty,0)}$ fulfills the hypotheses of Proposition \ref{thm:Wang_FNL} and is therefore a family of homothetically shrinking spheres or cylinders. In particular $M_0'$ is an affine subspace, hence $\mathcal T_\infty \Omega_0$ is an affine subspace.
\end{proof}

Finally we conclude $M_t$ is spherical or cylindrical.

\begin{proof}[Proof of Theorem \ref{thm:main_FNL}]
In case $\mathcal T_\infty \Omega_0 =\{0\}$, $M_t$ is compact, hence a sphere by \cite{Lang-Lyn20}. A streamlined version of the argument using results of this section proceeds as follows. Consider a sequence $a_j \to 0$ and set $M_t^j := a_j M_{a_j^{-2} t}$, $ t \leq 0$. Appealing to Lemmas \ref{lem:type-I_envelope_FNL}, \ref{lem:lower-speed_FNL} and \ref{lem:displacement_FNL}, and using the Krylov and Schauder estimates, we may pass to a subsequence such that
\[\{M_t^j\}_{t \in (-\infty, 0)} \to \{M_t'\}_{t \in (-\infty,0)}\]
in $C^\infty_{\loc}(\mathbb{R}^{n+1} \times (-\infty, 0))$. Moreover, we have 
\[M_t' \subset D(0, K \sqrt{-t}), \qquad B(0, c_K \sqrt{-t})\]
for all $t < 0$. The first inclusion implies $M_t'$ is compact.

This lets us conclude $M_t$ is uniformly convex. If not, there is a sequence of times $t_j \to -\infty$ and points $x_j \in M_{t_j}$ such that $\tfrac{\lambda_1}{H}(x_j,t_j) \to 0$. Choosing $a_j = 1/\sqrt{-t_j}$ we find $M_t'$ has a zero principal curvature at time $t = -1$, but since $\gamma$ admits a splitting theorem this implies $M_t$ splits off a Euclidean factor, contradicting compactness. Thus there is an $\varepsilon >0$ such that $\lambda_1 \geq \varepsilon H$ on $M_t$ for all $t \leq 0$, and consequently, for any sequence $a_j$ we choose, the blow-down limit $M_t'$ satisfies $\lambda_1' \geq \varepsilon H'$. Appealing to Andrews' result \cite{And94_euclid} we conclude $M_t'$ contracts to a round point as $t \to 0$, and in particular,
\[\min_{M_t'} \bigg( \frac{\lambda_1'}{H'} - \frac{1}{n}\bigg) \to 0.\]
Since $M_t'$ approximates $a_j M_{a_j^{-2}t}$, it follows that for every $\delta >0$ there is a sequence $\tau_j \to - \infty$ such that 
\[\min_{M_{\tau_j}} \frac{\lambda_1}{H} \geq \frac{1}{n} - \delta.\]
If $\gamma$ is concave and inverse-concave this bound is preserved for times $t \geq \tau_j$ \cite{And07}, so in fact $\lambda_1 \geq (\tfrac{1}{n} - \delta) H$ for all $t \leq 0$. Since $\delta$ was arbitrary, $M_t$ is totally umbilic, hence a sphere. In case $\gamma$ is convex one may proceed similarly using $\lambda_1/G$ in place of $\lambda_1/H$. 

If instead $T_\infty \Omega_0$ is noncompact, Lemma \ref{lem:final_slice_FNL} implies $\mathcal T_\infty \Omega_0 = \Sigma$ for some affine subspace $\Sigma \subset \mathbb{R}^{n+1}$.  Hence we have the splitting  
\[M_t = \Sigma \times M_t^\perp, \qquad M_t^\perp := M_t \cap \Sigma^\perp,\]
and moreover,
\[\mathcal T_\infty \Omega_0^\perp = \mathcal T_\infty \Omega_0 \cap \Sigma^\perp = \Sigma \cap \Sigma^\perp = \{0\},\]
so $M_t^\perp$ is compact. If $n-\dim \Sigma =1$ then $M_t^\perp$ is a circle in $\mathbb{R}^2$ by \cite{Dask-Ham-Ses}, and if instead $n-\dim \Sigma \geq 2$, then as above $M_t^\perp$ is a sphere. In all cases we have exhibited $\{M_t\}_{t\in(-\infty,0]}$ as a family of shrinking spheres or cylinders. 
\end{proof}

{\appendix

\section{Splitting theorem}
\label{app:splitting}

Given a symmetric cone $\Gamma' \subset \mathbb{R}^n$, let us write $\sym(\Gamma')$ for the set of symmetric $n\times n$-matrices with eigenvalues in $\Gamma'$. Consider an admissible speed $\gamma \in C^\infty(\Gamma)$. We abuse notation and write $\gamma$ also for the smooth function on $\sym(\Gamma)$ sending $A$ to $\gamma(\lambda(A))$, where $\lambda: \sym(\Gamma) \to \Gamma$ is the eigenvalue map. Recall $\gamma$ is inverse-concave if $\Gamma_+ \subset \Gamma$ and the function
\[\gamma_*(\lambda) := \gamma(\lambda_1^{-1},\dots,\lambda_n^{-1})^{-1}\]
is concave in $\Gamma_+$. If in addition $\gamma_*$ is strictly concave in non-radial directions, we say $\gamma$ is {\bf strictly inverse-concave}.

A simple computation shows $\gamma$ is (strictly) inverse-concave if and only if
\[\gamma_\dagger(\lambda) := - \gamma(\lambda^{-1}, \dots, \lambda_n^{-1})\]
is (strictly) concave in $\Gamma_+$. The function $\gamma_\dagger$ induces a smooth $O(n)$-invariant function (also denoted $\gamma_\dagger$) on $\sym(\Gamma_+)$ via the eigenvalue map, namely $\gamma_\dagger(A) = -\gamma(A^{-1})$. Differentiating this identity one finds 
\begin{align*}
\ddot \gamma^{ij,kl}_\dagger(A^{-1}) B_{ij}^* B_{kl}^* ={}& - ( \ddot \gamma^{ij,kl}(A) + 2 \dot \gamma^{ik}(A)A^{-1}_{jl})B_{ij} B_{kl}
\end{align*}
for all $A \in \sym(\Gamma_+)$ and $B \in \sym(\mathbb{R}^n)$, where $B^* := A^{-1} B A^{-1}$. The function $\gamma_\dagger$ is concave with respect to matrices if and only if it is concave with respect to eigenvalues (see for example \cite{Lang14}[Section 2.2]), so $\gamma$ is inverse-concave if and only if 
\begin{align*}
( \ddot \gamma^{ij,kl}(A) + 2 \dot \gamma^{ik}(A)A^{-1}_{jl})B_{ij} B_{kl} \geq 0
\end{align*}
for all $A\in\sym(\Gamma_+)$ and $B \in \sym(\mathbb{R}^n)$. Similarly, $\gamma$ is strictly inverse-concave if and only if 
\begin{equation}
\label{eq:strict_inverse-concave}
( \ddot \gamma^{ij,kl}(A) + 2 \dot \gamma^{ik}(A)A^{-1}_{jl})B_{ij} B_{kl}  > 0
\end{equation}
for all $A\in\sym(\Gamma_+)$ and nonzero $B \in \sym(\mathbb{R}^n)$.

We define an admissible speed $\gamma \in C^\infty(\Gamma)$ to be \textbf{strictly inverse-concave on $\partial \Gamma_+$} if $\Gamma_+ \subset \Gamma$ and the following condition is met. For each $1 \leq m \leq n-1$ such that 
\[\Gamma_0^m := \Big\{\lambda_1 = \dots = \lambda_m = 0, \min_{m < i\leq n} \lambda_i >0\Big\} \subset \Gamma,\]
the function $\gamma^{(m)} (\lambda) := \gamma(0,\dots,0, \lambda)$ is strictly inverse-concave in $\Gamma_0^m$. For example, take $\Gamma = \{\lambda_i + \lambda_j >0\}$ and consider the two-harmonic mean
\[\gamma(\lambda) = \Big(\sum_{i<j} (\lambda_i + \lambda_j)^{-1}\Big)^{-1}.\]
In this case $\Gamma_0^{1}$ is in $\Gamma$, but $\Gamma_0^{2}$ is not. We have 
\[\gamma^{(1)}(\lambda) = \Big(\sum_{i} \lambda_i^{-1} + \sum_{i<j} (\lambda_i + \lambda_j)^{-1} \Big)^{-1}, \qquad \gamma^{(1)}_*(\lambda) = \sum_{i} \lambda_i + \sum_{i<j} (\lambda_i^{-1} + \lambda_j^{-1})^{-1}.\]
The function $\gamma^{(1)}_*$ is strictly concave in non-radial directions, so $\gamma$ is strictly inverse-concave on $\partial \Gamma_+$. Similarly, the $k$-harmonic means and the speeds considered in \cite{Lynch20} are all strictly inverse-concave on $\partial \Gamma_+$. 

For an easier proof of the following result in case $\gamma$ is convex, see \cite{And-Lang-McCoy14}[Corollary 1.2].

\begin{proposition}
\label{prop:SMP}
Fix $n \geq 2$ and consider an admissible speed $\gamma \in C^\infty(\Gamma)$ which is strictly inverse-concave on $\partial \Gamma_+$. Let $F:M\times[t_0 - T,t_0] \to \mathbb{R}^{n+1}$ be solution of $\partial_t F = -G\nu$, where $M$ is a connected smooth $n$-manifold and $G = \gamma(\lambda)$. Suppose the principal curvatures of $F(\cdot,t)$ are nonnegative and lie in a cone $\Gamma_0\Subset \Gamma$, and let $W$ denote the Weingarten map of $F(\cdot,t)$. If $\ker(W)$ has dimension $k$ at $(x_0,t_0)$, then it has dimension at least $k$ in $M\times [t_0-T,t_0]$. Moreover, if $v \in \ker (W)$, then $\nabla_v A = 0$. 
\end{proposition}

\begin{proof}
Label the principal curvatures of $F(\cdot,t)$ so that $\lambda_1 \leq \dots \leq \lambda_n$. We first show that if $\lambda_1(x_0,t_0) =0$ for some $x_0 \in M$ then $\lambda_1 \equiv 0 $ in $M\times[t_0-T,t_0]$. Consider a general point $(x,t) \in M\times [t_0-T, t_0]$, and suppose $\varphi$ is a $C^2$ function in a backward neighbourhood of $(x,t)$ such that $\lambda_1 \geq \varphi$ with equality at $(x,t)$. If $\{e_i\}$ is a local orthonormal frame such that $A(e_i,e_i) = \lambda_i$ at $(x,t)$, then at $(x,t)$ we have \cite{Lang14}[Theorem 4.18] (cf. \cite{And07} and \cite{Lang17}):
\begin{align*}
(\partial_t - \dot \gamma^{ij} \nabla_i \nabla_j) \varphi \geq{}& \dot\gamma^{ij} A^2_{ij} \varphi +\ddot \gamma^{ij,kl}\nabla_1 A_{ij} \nabla_1 A_{kl}+ 2 \sum_{\lambda_k>\lambda_1} (\lambda_k-\lambda_1)^{-1}\dot\gamma^{ij}   \nabla_1 A_{ik} \nabla_1 A_{jk}.
\end{align*}
Let us write $\nabla_1 A(x,t)=U+V$ where
\begin{align*}
U:={}& \nabla_1 A_{11} e^1 \otimes e^1 + \sum_{i>1} \nabla_1 A_{i1} e^i \otimes e^1 +  \sum_{j>1} \nabla_1 A_{1j} e^1 \otimes e^j,\\
V:={}& \sum_{i,j>1} \nabla_1 A_{ij} e^i \otimes e^j,
\end{align*}
and define $m$ to be the multiplicity of $\lambda_1$ at $(x,t)$, so that at $(x,t)$ we have
\begin{align*}
(\partial_t - \dot \gamma^{ij} \nabla_i \nabla_j) \varphi \geq{}&\dot\gamma^{ij} A^2_{ij} \varphi +\ddot \gamma^{ij,kl}U_{ij}U_{kl} + 2 \ddot \gamma^{ij,kl} U_{ij}V_{kl} + \ddot \gamma^{ij,kl}V_{ij}V_{kl}\\
&+ 2  \sum_{i,j,k>m}\lambda_k^{-1} \dot\gamma^{ij}   \nabla_1 A_{ik} \nabla_1 A_{jk}.
\end{align*}

If $v$ is a unit tangent vector at $(x,t)$ with $A(v,v) = \lambda_1$, then $\nabla A(v,v) = \nabla \varphi(x,t)$. Therefore,  for all $1< i \leq m$ and $1 \leq j \leq n$, by the Codazzi equations and the polarisation identity,
\begin{align*}
\nabla_1 A_{ij}(x,t) = \nabla_j A_{1i}(x,t) = 0,
\end{align*}
hence $V = \sum_{i,j >m} \nabla_1 A_{ij}$. Consequently, at $(x,t)$ we have 
\begin{align*}
(\partial_t - \dot \gamma^{ij} \nabla_i \nabla_j) \varphi \geq{}&\dot\gamma^{ij} A^2_{ij} \varphi +\ddot \gamma^{ij,kl}U_{ij}U_{kl} + 2 \ddot \gamma^{ij,kl} U_{ij}V_{kl} \\
&+ \sum_{i,j,k,l > m} (\ddot \gamma^{ij,kl} + 2 \dot\gamma^{ik}\lambda_j^{-1} \delta_{jl})V_{ij}V_{kl}.
\end{align*}
Appealing to the Codazzi equations we obtain $|U|\leq C|\nabla A_{11}|$ for some $C= C(n)$, so since $\nabla A_{11} (x,t) = \nabla \varphi(x,t)$, at $(x,t)$ we have
\begin{align}
\label{eq:lambda_1_est}
(\partial_t - \dot \gamma^{ij} \nabla_i \nabla_j) \varphi \geq{}&\dot\gamma^{ij} A^2_{ij} \varphi - C G^{-1} |\nabla A| |\nabla \varphi| +\sum_{i,j,k,l > m} (\ddot \gamma^{ij,kl} + 2 \dot\gamma^{ik}\lambda_j^{-1} \delta_{jl})V_{ij}V_{kl}
\end{align}
where $C=C(n,\gamma,\Gamma_0)$. 

By assumption the function $\gamma^{(m)}$ is a strictly inverse-concave, so in light of \eqref{eq:strict_inverse-concave}, the final term on the right in \eqref{eq:lambda_1_est} is nonnegative (and indeed strictly positive unless $V = 0$). The point $(x,t)$ was arbitrary, so by the strong maximum principle for viscosity solutions to parabolic equations,\footnote{See for example \cite{DaLio}} since $\lambda_1$ takes its spacetime minimum at $(x_0,t_0)$, we have $\lambda_1 \equiv 0$ in $M\times [t_0-T,t_0]$. 

Next, given $(x,t) \in M\times[t_0-T,t_0]$, suppose $v \in \ker(W(x,t))$. We claim 
\[\nabla_v A(x,t) =0.\]
To see this, choose a local orthonormal frame $\{e_i\}$ such that $v = |v| e_1$ and $A(e_i,e_i) = \lambda_i$ at $(x,t)$. As before, let us write $\nabla_1 A(x,t) = U + V$, where 
\[V = \sum_{i,j >1} \nabla_1 A_{ij} e^i \otimes e^j = \sum_{i,j >m} \nabla_1 A_{ij} e^i \otimes e^j,\] 
$m$ being the multiplicity of $\lambda_1$ at $(x,t)$. Since $\lambda_1 \equiv 0$ we can apply \eqref{eq:lambda_1_est} with $\varphi \equiv 0$ to conclude 
\[0 = \sum_{i,j,k,l > m} (\ddot \gamma^{ij,kl} + 2 \dot\gamma^{ik}\lambda_j^{-1} \delta_{jl})V_{ij}V_{kl}\]
at $(x,t)$, hence $V=0$ by \eqref{eq:strict_inverse-concave} applied to $\gamma^{(m)}$. Since $\nabla A_{11} (x,t) = \nabla \varphi(x,t) =0$, $U=0$ as well, so it follows that
\[|v|^{-1}\nabla_v A(x,t) = \nabla_1 A(x,t) = U+V = 0.\]
The point $(x,t)$ was arbitrary, so we conclude that $\nabla_v A =0$ whenever $v \in \ker(W)$. 

For each $1 \leq k \leq n-1$, define $\Lambda_k := \lambda_1 + \dots + \lambda_k$. To finish, we will show that if $\Lambda_k(\cdot,t_0)$ vanishes somewhere in $M$, then $\Lambda_k \equiv 0$ in $M\times[t_0-T,t_0]$. We have already seen this to be true for $k=1$, and the remaining cases follow by induction, as follows. Let $2 \leq k \leq n-1$ be fixed, suppose $\Lambda_{k-1} \equiv 0$ in $M\times[t_0-T,t_0]$, and also that 
\[\Lambda_{k}(x_0,t_0) = 0\] for some $x_0 \in M$. Consider a general $(x,t) \in M\times [t_0 - T, t_0]$ and let $\{e_i\}$ be a local orthonormal frame with $A(x,t)(e_i, e_i)=\lambda_i$. Then, in particular, $e_i \in \ker(W(x,t))$ for all $1\leq i \leq k-1$, and we consequently have
\begin{equation}
\label{eq:SMP_grad}
\nabla_i A(x,t) =0, \qquad 1 \leq i \leq k-1.
\end{equation}
If $\varphi$ is a $C^2$ function defined in a backward neighbourhood of $(x,t)$ such that $\Lambda_k \geq \varphi$ with equality at $(x,t)$, then at $(x,t)$ we have (again, this computation can be found in \cite{Lang14}[Theorem 4.18] and also \cite{Lang17}):
\begin{align*}
(\partial_t - \dot \gamma^{ij} \nabla_i \nabla_j) \varphi \geq{}& \dot\gamma^{ij} A^2_{ij} \varphi +\sum_{l = 1}^k\ddot \gamma^{pq,rs}\nabla_l A_{pq} \nabla_l A_{rs}\\
&+ 2 \sum_{l = 1}^k  \sum_{p > k, \, \lambda_p>\lambda_l} (\lambda_p-\lambda_l)^{-1}\dot\gamma^{ij}   \nabla_l A_{ip} \nabla_l A_{jp}\\
\geq{}& \dot\gamma^{ij} A^2_{ij} \varphi +\ddot \gamma^{pq,rs}\nabla_k A_{pq} \nabla_k A_{rs}+ 2 \sum_{\lambda_p>\lambda_k} (\lambda_p-\lambda_k)^{-1}\dot\gamma^{ij}   \nabla_k A_{ip} \nabla_k A_{jp},
\end{align*}
where in the last line we have used \eqref{eq:SMP_grad}. If $e_k(x,t) \in \ker(W(x,t))$ we have $\nabla_k A(x,t)=0$, hence 
\begin{align*}
(\partial_t - \dot \gamma^{ij} \nabla_i \nabla_j) \varphi \geq{}0
\end{align*}
at $(x,t)$, so assume $\lambda_k(x,t) >0$. Write $\nabla_k A(x,t)$ as $U + V$ where
\begin{align*}
U &= \nabla_k A_{kk} e^k \otimes e^k+ \sum_{k < p \leq n} \nabla_k A_{pk} e^p \otimes e^k + \sum_{k<q\leq n} \nabla_k A_{kq} e^k\otimes e^q,\\
V &= \sum_{p,q >k} \nabla_k A_{pq} e^p \otimes e^q.
\end{align*}
Using $\nabla A_{kk} (x,t) = \nabla \varphi (x,t)$ and the Codazzi equations we estimate
\begin{align*}
(\partial_t - \dot \gamma^{ij} \nabla_i \nabla_j) \varphi \geq{}& \dot\gamma^{ij} A^2_{ij} \varphi - C|\nabla A||\nabla \varphi|+\sum_{p,q,r,s>k} (\ddot \gamma^{pq,rs}+ 2\dot\gamma^{pr}\lambda_q^{-1}\delta_{qs})V_{pq}V_{rs}
\end{align*}
at $(x,t)$, where $C = C(n,\gamma,\Gamma_0)$. As in the $k=1$ case, inverse-concavity of $\gamma^{(k)}$ and \eqref{eq:strict_inverse-concave} now imply
\begin{align*}
(\partial_t - \dot \gamma^{ij} \nabla_i \nabla_j) \varphi \geq{}& \dot\gamma^{ij} A^2_{ij} \varphi - C|\nabla A||\nabla \varphi|
\end{align*}
at $(x,t)$. Since this inequality holds at every point $(x,t) \in M\times[t_0 -T, t_0]$ where $\Lambda_k$ admits a lower support $\varphi$, and $\Lambda_k$ takes its global spacetime minimum at $(x_0,t_0)$, the strong maximum principle implies $\Lambda_k \equiv 0$ in $M\times [t_0 -T, t_0]$. 
\end{proof}

Standard arguments now imply the following splitting theorem. We borrow heavily from \cite{Bourni-Lang}[Theorem 5.1] (cf. \cite{Ham86}[Section 8] and \cite{Huisk-Sin99a}[Theorem 4.1]).

\begin{proposition}
\label{prop:split}
Fix an admissible speed $\gamma\in C^\infty(\Gamma)$ which is strictly inverse-concave on $\partial \Gamma_+$.  Let $F:M\times (- T,0] \to \mathbb{R}^{n+1}$ be a complete solution of $\partial_t F = -G\nu$, where $M$ is connected, and suppose the principal curvatures of $F(\cdot,t)$ are nonnegative and lie in $\Gamma_0 \Subset \Gamma$. Then there is a constant $\delta > 0$ such that $\dim(\ker W)$ is constant in $M\times (-\delta,0]$. In case $\dim(\ker W) = m \geq 1$ in $(-\delta,0]$, $M_t := F(M,t)$ splits as a product 
\[M_t = \mathbb{R}^m\times N_t, \qquad t \in (-\delta, 0], \]
where $N_t$ is a smooth hypersurface in $\mathbb{R}^{n-m+1}$ with positive second fundamental form. If $T=\infty$ and $M_t$ is the boundary of a convex domain for each $t \leq 0$ we can take $\delta = \infty$. 
\end{proposition}

\begin{proof}
Let $x_0 \in M$ be a point where $\dim(\ker(W(\cdot, 0))$ attains its maximum, which we denote by $m$. Then Proposition \ref{prop:SMP} implies $\dim(\ker W) \geq m$ in $M\times (-T,0]$, and there is a time $\delta >0$ such that $\lambda_{m+1} >0$ in $M \times (-\delta,0]$. In particular, since $\dim(\ker W)$ is constant in $M\times (-\delta,0]$, the subspace $\ker W$ is smooth over $M\times(-\delta,0]$.

Let us restrict attention to times $t \in (-\delta, 0]$. We claim $F_* \ker W$ is a parallel subspace of $T\mathbb{R}^{n+1}$. To see this, note by Proposition \ref{prop:SMP} that if $v$ is a smooth vector field in $\ker W$ then $\nabla_v A =0$, so by the Codazzi equations
\[0 = \nabla_v A(w,u) = \nabla_u A(v,w)= -A(\nabla_u v, w) \]
for all $u$ and $w$, hence $\nabla_u v \in \ker W$. Thus, for any $u$ tangent to $M$ we compute
\[D_{F_*u} \, F_* v  = F_* \nabla_{u} v  \in F_* \ker W,\]
where $D$ is the Euclidean connection on $\mathbb{R}^{n+1}$. This shows $F_* \ker W$ is parallel in space. We also have 
\[\partial_t F_*v = F_* \partial_t v ,\]
so since
\[- W(\partial_t v) = (\partial_t W)(v) = \dot\gamma^{ij} \nabla_i \nabla_j W (v) = 0,\]
$F_*\ker W$ is parallel in time. 

We conclude that  $\mathbb{R}^{n+1}$ can be split as $\mathbb{R}^{m} \times \mathbb{R}^{n-m+1}$ in such a way that $F_* \ker W$ is orthogonal to $\mathbb{R}^{n-m+1}$ for all $t \in (-\delta,0]$. Since $F_* \ker W$ is parallel, given any nonzero $v \in \ker W{(x,t)}$, the geodesic in $M_{t}$ tangent to $F_* v$ at $F(x,t)$ is a straight line in $\mathbb{R}^{n+1}$ orthogonal to $\mathbb{R}^{n-m+1}$. It follows that $M_t = \mathbb{R}^m \times N_t$ for each $t \in (-\delta,0]$, where 
\[N_t:=M_t \cap \mathbb{R}^{n-m+1}.\]
Since $\lambda_{m+1}$ is positive in $M\times(-\delta,0]$, the hypersurfaces $N_t$ all have positive second fundamental form. 

Now suppose $F$ is defined for all times $t \in (-\infty,0]$, $M_t$ is the boundary of a convex domain $\Omega_t$ for each $t\leq 0$, and $M_t$ has a zero principal curvature for some $t \leq 0$ . Proposition \ref{prop:SMP} implies $\dim(\ker W)$ is nonincreasing in time, so there is a $T<0$ such that $\dim (\ker W)$ is constant in $M\times(-\infty,T]$. Moreover, as we have just seen, $M_t = \mathbb{R}^m \times N_t$ for all $t \leq T$ and some fixed $0 \leq m \leq n-1$. Translating the solution if necessary, we may assume the origin is contained in $\Omega_0$. Then $\mathbb{R}^m$ is a subset of $\Omega_t$ for all $t\leq T$. Let $T'$ be the supremum of all those $t\in[T,0]$ such that $\mathbb{R}^m \subset \Omega_t$. If $T' < 0$ then there is a point $x_0 \in \mathbb{R}^m$ such that $x_0 \in \Omega_t$ for all $t < T'$ but $x_0 \not \in \Omega_{T'}$. On the other hand there is a ball $B(0,r) \subset \Omega_0 \subset \Omega_t$ for all $t \leq 0$, so by convexity $B(x,r) \subset \Omega_t$ for all $x \in \mathbb{R}^m$ and $t  <T'$. In particular, $B(x_0,r) \subset \Omega_t$ for all $t <T'$, but this contradicts the avoidance principle, so in fact $\mathbb{R}^m \subset \Omega_t$ for all $t \leq 0$. By convexity, the splitting $M_t = \mathbb{R}^m \times N_t$ persists for times $t \in [T,0]$.
\end{proof}
}

\bibliographystyle{alpha}
\bibliography{references}

\end{document}